\documentclass[11pt]{article}
\usepackage[margin=1in]{geometry}
\usepackage[utf8]{inputenc}
\usepackage[english]{babel}
\usepackage[T1]{fontenc}

\usepackage{verbatim}
\usepackage{relsize}
\usepackage{amsmath,amsfonts,amssymb,amsthm,amstext,braket,mathtools,floatrow, caption,subcaption, algorithm,algorithmic, graphics,enumitem,booktabs}
\usepackage{blkarray}
\usepackage[colorlinks=true, urlcolor=blue, linkcolor=blue, citecolor=blue]{hyperref}
\usepackage{tikz}
\usepackage{comment}
\usepackage{xcolor}
\usepackage{cleveref}
\usepackage[normalem]{ulem}
\usetikzlibrary{positioning,calc}
\usepackage[authoryear]{natbib}

\numberwithin{equation}{section}
\theoremstyle{plain}
\newtheorem*{theorem*}{Theorem}
\newtheorem{theorem}{Theorem}
\numberwithin{theorem}{section}

\newtheorem{lemma}[theorem]{Lemma}
\newtheorem{corollary}[theorem]{Corollary}
\newtheorem{conjecture}[theorem]{Conjecture}

\theoremstyle{definition}
\newtheorem{definition}[theorem]{Definition}
\newtheorem{remark}[theorem]{Remark}
\newtheorem{example}[theorem]{Example}

\newcommand{\sign}{\text{sign}}
\newcommand{\rk}{\text{rk}}
\newcommand{\bR}{\mathbb{R}}
\newcommand{\im}{\text{Im}}
\newcommand{\cO}{\mathcal{O}}
\newcommand{\cM}{\mathcal{M}}
\newcommand{\cL}{\mathcal{L}}
\renewcommand{\O}{\Omega}

\newcommand{\pa}{\text{pa}}
\newcommand{\des}{\text{des}}
\newcommand{\bfP}{\mathbf{P}}
\newcommand{\oLambda}{\overline{\Lambda}}
\newcommand{\olambda}{\overline{\lambda}}

\newcommand{\opa}{\overline{\text{pa}}}
\DeclareMathOperator{\PD}{\mathit{PD}}
\newcommand{\can}{\text{can}}
\newcommand{\lat}{\text{lat}}
\newcommand{\elr}{\text{elr}}
\newcommand{\lr}{\text{lr}}

\newcommand{\lp}{\text{lp}}
\newcommand{\flow}{\text{flow}}
\newcommand{\reg}{\text{reg}}

\allowdisplaybreaks

\newcommand{\algorithmicbreak}{\textbf{break}}
\newcommand{\BREAK}{\STATE \algorithmicbreak}

\newcommand{\footremember}[2]{%
    \footnote{#2}
    \newcounter{#1}
    \setcounter{#1}{\value{footnote}}%
}

\title{Trek-Based Parameter Identification for Linear Causal Models With Arbitrarily Structured Latent Variables}

\author{
Nils Sturma\footremember{nils}{Technical University of Munich, Germany; \href{mailto:nils.sturma@tum.de}{nils.sturma@tum.de}}
\and 
Mathias Drton\footremember{mathias}{Technical University of Munich and Munich Center for Machine Learning, Germany;
\href{mailto:mathias.drton@tum.de}{mathias.drton@tum.de}}  
}

\date{ }
\begin{document}
\maketitle

\begin{abstract}
We develop a criterion to certify whether causal effects are identifiable in linear structural equation models with latent variables. Linear structural equation models correspond to directed graphs whose nodes represent the random variables of interest and whose edges are weighted with linear coefficients that correspond to direct causal effects. In contrast to previous identification methods, we do not restrict ourselves to settings where the latent variables constitute independent latent factors (i.e., to source nodes in the graphical representation of the model). Our novel latent-subgraph criterion is a purely graphical condition that is sufficient for identifiability of causal effects by rational formulas in the covariance matrix. To check the latent-subgraph criterion, we provide a sound and complete algorithm that operates by solving an integer linear program. While it targets effects involving observed variables, our new criterion is also useful for identifying effects between latent variables, as it allows one to transform the given model into a simpler measurement model for which other existing tools become applicable.
\end{abstract}

\section{Introduction}
The problem of identifiability of causal effects amounts to studying whether it is feasible to infer cause-effect relationships under clearly detailed assumptions about the process that generates the observed data. Determining whether causal effects are identifiable is crucial for any downstream task, such as robust estimation of effects, inferring interventional distributions, or answering counterfactual queries \citep{pearl2009causality}.  Much attention has been paid to models in which latent variables solely act as confounders that induce correlation between errors in structural equations \citep{kumor2020efficient, weihs2017determinantal, foygel2012halftrek}. However, for drawing inference in complex data where one is also interested in causal relations among latent variables, the problem of causal effect identifiability is in large parts unsolved.

In this paper, we study linear causal models given by flexible structural equation models that explicitly consider latent variables. The causal parameters of interest are the linear coefficients appearing in the equations. The dominant approach in state-of-the-art methods to handle settings with explicitly modeled latent variables is to transform the model into a \emph{canonical model}, where each latent variable is an independent factor \citep{tramontano2024causal, salehkaleybar2020learning, hoyer2008estimation}. This comes with several drawbacks. It removes all causal relations between latent variables and thus makes inference of these relations infeasible. However, in a broad range of applied sciences \citep{abbring2025wright, stoetzer2024causal, mayer2019causal, bollen2011three, sachs2005causal} as well as in the recently evolving field of causal representation learning \citep{saengkyongam2024identifying, sturma2023unpaired, squires2023linear, schoelkopf2021toward}, the causal relations between the latent variables are of most interest. Moreover, in contrast to what is usually presumed in the literature, we show that the transformed canonical model is not identical to the original model, and identifiability of causal effects can hold in one of them while it fails in the other. 

Concretely, we study linear structural equation models defined as follows. Let $G=(V,D)$ be a directed graph and let $X=(X_v)_{v \in V}$ be a collection of random variables that are indexed by the nodes $V$ of the graph. We suppose that all variables are related by the linear equation system
\begin{equation} \label{eq:full-sem}
    X = \Lambda^{\top}X+\varepsilon,
\end{equation}
where $\varepsilon=(\varepsilon_v)_{v\in V}$ is a collection of independent mean zero random variables with finite variance $\phi_v > 0$. The coefficient matrix $\Lambda=(\lambda_{vw})$ is sparse according to the graph $G$, that is, $\lambda_{vw}\neq 0$ only if $v \rightarrow w \in D$. Hence, each graph defines a different model via a different set of equations~\eqref{eq:full-sem}. Crucially, we allow that some of the random variables $X_v$ are latent. That is, the set of nodes is a disjoint union, denoted $V=\cO \sqcup \cL$, where nodes in $\cO$ index the observed variables and nodes in $\cL$ index the latent variables. We also say that $\cO$ are the \emph{observed nodes} and $\cL$ are the \emph{latent nodes}. 

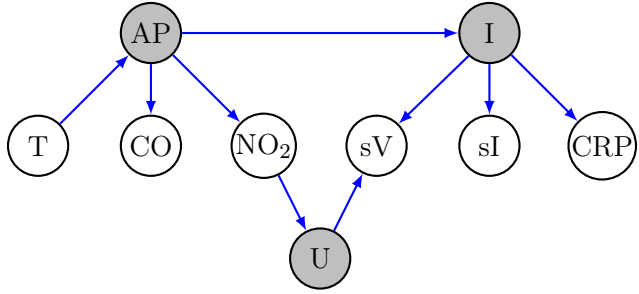
\begin{figure}[t]
    \centering
    \tikzset{
      every node/.style={circle, inner sep=0.15mm, minimum size=0.8cm, draw, thick, black, fill=white, text=black},
      every path/.style={thick}
    }
    \begin{tikzpicture}[align=center]
      \node[fill=lightgray] (h1) at (1.5,1.5) {AP};
      \node[fill=lightgray] (h2) at (6,1.5) {I};
      \node[fill=lightgray] (h3) at (3.75,-1.5) {U};
      \node[] (1) at (0,0) {T};
      \node[] (2) at (1.5,0) {CO};
      \node[] (3) at (3,0) {NO$_2$};
      \node[] (4) at (4.5,0) {sV};
      \node[] (5) at (6,0) {sI};
      \node[] (6) at (7.5,0) {CRP};

      \draw[blue] [-latex] (1) edge (h1);
      \draw[blue] [-latex] (3) edge (h3);
      \draw[blue] [-latex] (h1) edge (h2);
      \draw[blue] [-latex] (h1) edge (2);
      \draw[blue] [-latex] (h1) edge (3);
      \draw[blue] [-latex] (h2) edge (4);
      \draw[blue] [-latex] (h2) edge (5);
      \draw[blue] [-latex] (h2) edge (6);
      \draw[blue] [-latex] (h3) edge (4);
    \end{tikzpicture}
    \caption{Graph corresponding to a study for the effect of air pollution on inflammation. Gray nodes correspond to latent random variables.}
    \label{fig:air-pollution}
\end{figure}

\begin{example}
    We consider an augmented version of an example from \citet{baja2013structural}, which pertains to the effect of air pollution (AP) on inflammation (I); see Figure \ref{fig:air-pollution}. The inflammation of different individuals is measured via C-reactive protein levels (CRP), soluble vascular cell adhesion molecule-1 levels (sV) and soluble intracellular adhesion molecule-1 levels (sI), while at the same time air pollution is measured via carbon monoxide (CO) and nitrogen dioxide (NO$_2$) concentration. In classical analysis of structural equation models, it is impossible to include observed nodes that have an effect on the latent nodes. However, traffic density (T) certainly has a direct effect on air pollution and it can be insightful to include such a variable. Moreover, one might worry about indirect effects of high concentrations of air pollution measurements on the measurement of inflammation and therefore include further latent nodes (U).
\end{example}

The coefficients $\lambda_{vw}$ appearing in the equations in \eqref{eq:full-sem} correspond to the edges $v \rightarrow w$ in the graph, and are also known as \emph{direct causal effects}. Identifiability then refers to the question whether it is possible to recover the direct causal effects from the joint distribution of the  observed variables $X_{\cO}=(X_v)_{v \in \cO}$. Most interest is given to methods that only assume existence of the covariance matrix and provide explicit identification formulas for the direct effects, as this allows for simple estimation and inference \citep{henckel2024graphical, barber2022halftrek, kumor2020efficient, brito2006graphical}.

However, in flexible models without restrictions on the latent variables a main complication arises through the fact that even causal effects between observed variables can propagate through latent variables. For example, in Figure~\ref{fig:air-pollution} traffic density has an indirect causal effect on carbon monoxide concentration that is mediated by the latent variable air pollution. We denote by   $\olambda_{vw}$ the \emph{semi-direct effect} from $v$ to $w$ that is given by the sum of the direct effect plus indirect effects that propagate through latent nodes. To make this precise, given two sets of observed nodes $A,B \subseteq V$, we write $\Lambda_{A,B}$ for the submatrix that is obtained by taking rows indexed in $A$ and columns indexed in $B$. The matrix of semi-direct effects $\oLambda=(\olambda_{vw})_{v,w \in \cO}$ is then formally defined by
\begin{equation} \label{eq:semi-direct-effects}
    \oLambda = \Lambda_{\cO,\cO} +  \Lambda_{\cO,\cL} (I - \Lambda_{\cL,\cL})^{-1} \Lambda_{\cL,\cO},
\end{equation}
where entries in $\Lambda_{\cO,\cO}$ correspond to direct effects given by edges, and entries in $\Lambda_{\cO,\cL} (I - \Lambda_{\cL,\cL})^{-1} \Lambda_{\cL,\cO}$ correspond to directed paths that only visits latent nodes. Now, denote by $\Sigma$ the observed covariance matrix of $X_{\cO}$. A key observation is that $\Sigma$ factorizes as
\begin{equation} \label{eq:observed-covariance}
    \Sigma = (I - \oLambda)^{-\top} \Omega (I - \oLambda)^{-1}, 
\end{equation}
where $\Omega$ is a covariance matrix in a submodel. The submodel is defined by the subgraph $G_{\lat}$ on the same set of nodes $V=\cO \sqcup \cL$ obtained from removing all edges with tail being an observed node.   
A detailed derivation of the factorization \eqref{eq:observed-covariance} is given in Section~\ref{sec:preliminaries}. Equation~\eqref{eq:observed-covariance} is equivalent to $\Omega = (I - \oLambda)^{\top} \Sigma (I - \oLambda)$, which implies that $\Omega$ is also identifiable, whenever $\oLambda$ is identifiable from the covariance matrix $\Sigma$. Since $\Omega$ corresponds to the observed covariance matrix in much simpler measurement models that were already considered in \citet{bollen1989structural}, many tools are available to further investigate and identify the causal effects between latent variables \citep{sturma2025matching, xie2020generalized, bollen1989structural}. In short, identifying the semi-direct effect matrix $\oLambda$ allows us to transform our initial complex models to a simpler measurement model. The main result of this paper is the \emph{latent-subgraph criterion},  a graphical criterion for identifiability of $\oLambda$. To our knowledge, it is the first identifiability criterion for semi-direct effects without assuming that the latent nodes are source nodes. \looseness=-1

\begin{example} \label{ex:identifiying-latent-effects}
We continue the earlier example on air pollution shown in the graph in Figure~\ref{fig:air-pollution}. The semi-direct effect matrix is given by
\[
\oLambda = 
\begin{pmatrix}
    0 & \lambda_{\text{T}, \text{AP}} \lambda_{\text{AP}, \text{CO}}  & \lambda_{\text{T}, \text{AP}} \lambda_{\text{AP}, \text{NO}_2} & \lambda_{\text{T}, \text{AP}}\lambda_{\text{AP}, \text{I}} \lambda_{\text{I}, \text{sV}} & \lambda_{\text{T}, \text{AP}}\lambda_{\text{AP}, \text{I}} \lambda_{\text{I}, \text{sI}} & \lambda_{\text{T}, \text{AP}}\lambda_{\text{AP}, \text{I}} \lambda_{\text{I}, \text{CRP}}  \\
    0 & 0 & 0 & 0 & 0 & 0 \\
    0 & 0 & 0 & \lambda_{\text{NO}_2, \text{U}}\lambda_{\text{U}, \text{sV}} & 0 & 0 \\
    0 & 0 & 0 & 0 & 0 & 0 \\
    0 & 0 & 0 & 0 & 0 & 0 \\
    0 & 0 & 0 & 0 & 0 & 0 \\
\end{pmatrix}.
\]
The new latent-subgraph criterion is able to certify that the whole matrix $\oLambda$ is identifiable by rational formulas in the entries of the covariance matrix. For example, the semi-direct effect $\olambda_{\text{T}, \text{CO}}= \lambda_{\text{T}, \text{AP}} \lambda_{\text{AP}, \text{CO}}$ from traffic density to carbon monoxide concentration is given by the standard regression coefficient $\Sigma_{\text{T}, \text{CO}}/ \Sigma_{\text{T}, \text{T}}$. Now, the matrix $\Omega=(\omega_{uv})$ in this example corresponds to the observed covariance  matrix of the model given by the graph where we removed the two edges $\text{T} \to \text{AP}$ and $\text{NO}_2 \to \text{U}$.  It is then possible to identify the direct effect between the latent variables air pollution (AP) and inflammation (I) up to sign via the simple formula in the entries of $\Omega$ given by \looseness=-1
\begin{equation} \label{eq:formula-latent-effect}
    \sqrt{\frac{\omega_{\text{CO}, \text{CRP}}\omega_{\text{NO}_2, \text{sI}}}{\omega_{\text{CO}, \text{NO}_2}\omega_{\text{sI}, \text{CRP}} - \omega_{\text{CO}, \text{CRP}} \omega_{\text{NO}_2, \text{sI}}}} = \sqrt{\lambda_{\text{AP}, \text{I}}^2} = |\lambda_{\text{AP}, \text{I}}|,
\end{equation}
when fixing the scaling of the latent variables; see Appendix~\ref{sec:latent-effects} for a detailed derivation. We are not aware of any method for identification in the literature that would be able to certify identifiability of this effect. Note that identification of  $|\lambda_{\text{AP}, \text{I}}|$ in the simpler measurement model after removing the edges  $\text{T} \to \text{AP}$  and $\text{NO}_2 \to \text{U}$ is also given by the identification rules in \citet[Chapter 8]{bollen1989structural}. \looseness=-1
\end{example}

Our new latent-subgraph criterion for identification of the semi-direct effect matrix $\oLambda$ operates on the original, non-transformed model given in~\eqref{eq:full-sem} where any variable may be latent. Existing work mainly considers canonical models where direct effects can \emph{not} propagate through latent variables \citep{tramontano2024causal, kaltenpoth2023nonlinear, barber2022halftrek, salehkaleybar2020learning}.   Recently, some authors have also started to consider unrestricted latent structures. For example, \citet{ankan2023combining} considers identification of \emph{single} effects $\lambda_{vw}$, which is, however, also based on a transformation of the original model. \citet{dong2024parameter} derived a condition for the full matrix $\Lambda$, but this condition applies only in restricted settings where each latent variable must have at least two pure children.

Let $v \in \cO$ and $P \subseteq \cO$, and suppose we are interested in identifying the semi-direct effects $\oLambda_{P,v}=(\olambda_{pv})_{p \in P}$ from the nodes in $P$ into the node $v$. Our approach is to find a linear equation system 
\[
    \begin{pmatrix}
        \Sigma_{Y,P}  & \Omega_{Y,Z}
    \end{pmatrix}
    \begin{pmatrix}
        \oLambda_{P,v} \\
        \psi
    \end{pmatrix}    
     = \Sigma_{Y,v},
\]
where $Z \subseteq \cO \setminus P $ is a set of nodes of suitable size such that $(\Sigma_{Y,P} \,\,\, \Omega_{Y,Z})$ is a square matrix consisting of two blocks that are submatrices of  $\Sigma$ and  $\Omega$, and $\psi$ is a real-valued vector of suitable size. If $\Omega_{Y,Z}$ is already identified, and one additionally makes sure that $(\Sigma_{Y,P} \,\,\, \Omega_{Y,Z})$ has nonzero determinant, it follows that we can solve for the semi-direct effects $\oLambda_{P,v}$. Showing that the determinant of $(\Sigma_{Y,P} \,\,\, \Omega_{Y,Z})$ is nonzero requires new proof techniques. The well-known trek separation criterion is a  graph-theoretic characterization of when determinants of submatrices $\Sigma_{Y,P}$ of the covariance matrix are zero \citep{sullivant2010trek}. However, it is not applicable to block-matrices, where some blocks correspond to covariance matrices of subgraphs. We give a novel graph-theoretic criterion that we call \emph{trek separation in subgraphs} and that allows us to certify that the determinant of matrices of the form $(\Sigma_{Y,P} \,\,\, \Omega_{Y,Z})$ is nonzero.  It generalizes the usual notion of trek separation and has potential applications in theory for graphical models that goes beyond the work in this paper. 

Finally, we show that our novel latent-subgraph criterion for identification of semi-direct effects can be efficiently checked by repeatedly solving a certain integer linear program. It extends the standard maximum flow problem to settings where some flows are required to use only edges in a subgraph \citep{cormen2009introduction}.

The organization of the paper is as follows. In Section~\ref{sec:preliminaries}, we derive the factorization~\eqref{eq:observed-covariance}, we provide a precise definition of rational identifiability, and we introduce necessary graphical concepts. In Section~\ref{sec:lsc} we present our main result, the latent-subgraph criterion. In Section~\ref{sec:computation}, we provide an algorithm for checking our criterion by setting up an integer linear program. In Section~\ref{sec:canonical}, we discuss differences to the canonical model considered in previous research. Finally, in Section~\ref{sec:trek-separation}, we present our most important technical tool, the trek separation in subgraphs. The appendix contains all technical proofs (Appendix~\ref{sec:proof-main-result} and Appendix~\ref{sec:other-proofs}) and additional examples (Appendix~\ref{sec:examples}).

\section{Preliminaries} \label{sec:preliminaries}
In this section, we derive the factorization of the observed covariance matrix displayed in~\eqref{eq:observed-covariance}, we rigorously introduce rational identifiability of semi-direct effects and we introduce treks.

\subsection{Factorization of the Covariance Matrix}
By definition of the linear structural equation model in~\eqref{eq:full-sem}, the observed and latent random vectors $X_{\cO}$ and $X_{\cL}$ satisfy 
\begin{equation}\label{eq:block-sem}
\begin{aligned}
    X_{\cO} &= \Lambda_{\cO,\cO}^{\top}X_{\cO} + \Lambda_{\cL,\cO}^{\top} X_{\cL}+ \varepsilon_{\cO} \,\,\, \text{ and } \\
    X_{\cL} &=  \Lambda_{\cO,\cL}^{\top} X_{\cO} + \Lambda_{\cL,\cL}^{\top}X_{\cL} + \varepsilon_{\cL},
\end{aligned}
\end{equation}
where we abbreviate $\Lambda_{A,B}^{\top}=(\Lambda_{A,B})^{\top}$. Solving the upper equation for $X_{\cO}$ and the lower equation for $X_{\cL}$ yields
\[
    X_{\cO} = (I - \Lambda_{\cO,\cO})^{-\top}(\Lambda_{\cL,\cO}^{\top}X_{\cL} + \varepsilon_{\cO}) \quad \text{and} \quad X_{\cL} = (I - \Lambda_{\cL,\cL})^{-\top}(\Lambda_{\cO,\cL}^{\top}X_{\cO} + \varepsilon_{\cL}),
\]
where $I$ denotes identity matrices of suitable sizes. For illustrational purposes, we assume here that all inverses exist.  Plugging the equation for $X_{\cL}$ into the upper equation in~\eqref{eq:block-sem} gives
\begin{align*}
    X_{\cO} &= \Lambda_{\cO,\cO}^{\top}X_{\cO} + \Lambda_{\cL,\cO}^{\top} (I - \Lambda_{\cL,\cL})^{-\top}(\Lambda_{\cO,\cL}^{\top}X_{\cO} + \varepsilon_{\cL})+ \varepsilon_{\cO}, \\
    \iff X_{\cO} &=  \oLambda^{\top} X_{\cO} + \Lambda_{\cL,\cO}^{\top}  (I -\Lambda_{\cL,\cL})^{-\top} \varepsilon_{\cL} + \varepsilon_{\cO},\\
    \iff X_{\cO} &= (I - \oLambda)^{-\top} [\Lambda_{\cL,\cO}^{\top}  (I -\Lambda_{\cL,\cL})^{-\top} \varepsilon_{\cL} + \varepsilon_{\cO}],
\end{align*}
where $\oLambda=\Lambda_{\cO,\cO} +  \Lambda_{\cO,\cL} (I - \Lambda_{\cL,\cL})^{-1} \Lambda_{\cL,\cO}$ is the matrix of semi-direct effects from~\eqref{eq:semi-direct-effects}.  
We conclude that the observed covariance matrix is given by
\begin{align} \label{eq:sigma-2}
    \Sigma &=  (I - \oLambda)^{-\top} \Omega (I - \oLambda)^{-1},
\end{align}
where $\Omega$ is defined as the covariance matrix of $\Lambda_{\cL,\cO}^{\top}  (I -\Lambda_{\cL,\cL})^{-\top} \varepsilon_{\cL} + \varepsilon_{\cO}$. That is,
\[
\Omega := \Lambda_{\cL,\cO}^{\top}  (I -\Lambda_{\cL,\cL})^{-\top} \Phi_{\cL,\cL} (I -\Lambda_{\cL,\cL})^{-1} \Lambda_{\cL,\cO} + \Phi_{\cO,\cO},
\]
where $\Phi$ is the diagonal covariance matrix of the noise vector $\varepsilon$ with diagonal entries $\phi_v > 0$. Now, for the purpose of identifiability, we may identify the linear structural equation model with a set of  covariance matrices of the form~\eqref{eq:sigma-2}. 
%Now, recall that each graph defines a different linear structural equation model by defining the sparsity pattern in $\Lambda$. . 
To formally define this, we first introduce the necessary notation. Let $G=(\cO \sqcup \cL, D)$ be a graph and consider two subsets of nodes $U,W \subseteq \cO \sqcup \cL$. We denote by $D_{U,W}=D \cap (U \times W)$ the corresponding subset of edges that point from nodes in $U$ to nodes in $W$. We then write $\mathbb{R}^{D_{U,W}}$ for the space of real $|U| \times |W|$ matrices $\Lambda=(\lambda_{vw})$ with support $D_{U,W}$, that is, $\lambda_{uw}=0$ if $u \rightarrow w \not\in D_{U,W}$.  We denote $\mathbb{R}^{D}_{\reg} $ the subset of matrices $\Lambda \in \mathbb{R}^{D}$ for which both $I-\oLambda$ and $I-\Lambda_{\cL,\cL}$ are invertible. Finally, $\PD(p)$ is the cone of positive definite $p \times p $ matrices and $\mathbb{R}^{p}_{>0} \subseteq \PD(p)$ is the subset of diagonal positive definite matrices.

\begin{definition} \label{def:model}
The \emph{covariance model} given by a graph $G=(\cO \sqcup \cL, D)$ is the image $\cM(G)=\im(\tau_G)$ of the parametrization
\begin{equation} \label{eq:parametrization}
\begin{aligned} 
\tau_G: \mathbb{R}^D_{\reg} \times \mathbb{R}^{|\mathcal{O}| + |\cL|}_{>0}  &\longmapsto \PD(|\mathcal{O}|)\\
(\Lambda, \Phi) &\longmapsto (I - \oLambda)^{-\top} \Omega (I - \oLambda)^{-1},
\end{aligned}
\end{equation}
where the matrices $\oLambda =(\olambda_{vw}) \in \mathbb{R}^{D_{\cO,\cO}}$ and $\Omega=(\omega_{vw}) \in \PD(|\cO|)$ are given by
\begin{align*}
\oLambda &= \Lambda_{\cO,\cO} + \Lambda_{\cO,\cL} (I - \Lambda_{\cL,\cL})^{-1} \Lambda_{\cL,\cO} \quad \text{and}\\
\Omega &= \Lambda_{\cL,\cO}^{\top}  (I -\Lambda_{\cL,\cL})^{-\top} \Phi_{\cL,\cL} (I -\Lambda_{\cL,\cL})^{-1} \Lambda_{\cL,\cO} + \Phi_{\cO,\cO}.
\end{align*}
\end{definition}

\begin{example}
For the graph in Figure~\ref{fig:running-example} (a), we explicitly state the matrices $\Lambda$, $\Phi$, $\oLambda$, $\Omega$ and $\Sigma$. Due to space constraints, this is deferred to Appendix~\ref{sec:example-parameter}.
\end{example}

For brevity, we let $\Theta_G=\mathbb{R}^D_{\reg} \times \mathbb{R}^{|\cO| + |\cL|}_{>0}$ be the domain of the parametrization $\tau_G$. A crucial observation is that the matrix $\Omega$ lies in a simpler submodel. Define  $G_{\lat}=(\cO \sqcup \cL,D_{\lat})$ to be the \emph{latent subgraph} of $G$ obtained from deleting all edges with tail being an observed node, that is, 
\begin{align} \label{eq:latent-subgraph}
    D_{\lat} &:= \{v \rightarrow w \in D: v \in \cL, w \in \cO\} \cup \{v \rightarrow w \in D: v,w \in \cL\}. 
\end{align}
The following lemma confirms that the matrix $\Omega$ lies in the covariance model $\cM(G_{\lat})$ given by the latent subgraph.

\begin{lemma}
Consider a graph $G=(\cO \sqcup \cL, D)$. Then, the covariance model given by the latent subgraph is 
\begin{align*}
    \cM(G_{\lat}) = \{\Sigma \in \PD(|\cO|): \ &\Sigma =  \Lambda_{\cL,\cO}^{\top}  (I -\Lambda_{\cL,\cL})^{-\top} \Phi_{\cL,\cL} (I -\Lambda_{\cL,\cL})^{-1} \Lambda_{\cL,\cO} + \Phi_{\cO,\cO} \\
    &\textrm{for some } (\Lambda, \Phi) \in \Theta_G\}.
\end{align*}
\end{lemma}
\begin{proof}
Observe that every $\Lambda' \in \mathbb{R}^{D_{\lat}}_{\reg}$ is of the form
\[
    \Lambda' = \begin{pmatrix}
    0 & 0 \\
    \Lambda_{\cL, \cO} & \Lambda_{\cL, \cL}
\end{pmatrix}
\]
for some $\Lambda \in \mathbb{R}^D_{\reg}$. By Definition~\ref{def:model}, this implies that $\Sigma \in \cM(G_{\lat})$ if and only if 
\[
    \Sigma = \Lambda_{\cL,\cO}^{\top}  (I -\Lambda_{\cL,\cL})^{-\top} \Phi_{\cL,\cL} (I -\Lambda_{\cL,\cL})^{-1} \Lambda_{\cL,\cO} + \Phi_{\cO,\cO} 
\]
for some $(\Lambda, \Phi) \in \Theta_G$.
\end{proof}

\begin{figure}[t]
\subfloat[]{
    \centering
    \tikzset{
      every node/.style={circle, inner sep=0.3mm, minimum size=0.45cm, draw, thick, black, fill=white, text=black},
      every path/.style={thick}
    }
    \begin{tikzpicture}[align=center]
      \node[fill=lightgray] (h1) at (-1,1) {$h_1$};
      \node[fill=lightgray] (h2) at (0.5,1) {$h_2$};
      
      \node[] (1) at (-2,0) {$v_1$};
      \node[] (2) at (-1,0) {$v_2$};
      \node[] (3) at (-0,0) {$v_3$};
      \node[] (4) at (1,0) {$v_4$};
      \node[] (5) at (2,0) {$v_5$};
      
      \draw[red, dashed] [-latex] (h1) edge (h2);
      \draw[red, dashed] [-latex] (h2) edge (2);
      \draw[red, dashed] [-latex] (h2) edge (3);
      \draw[red, dashed] [-latex] (h2) edge (4);
      \draw[red, dashed] [-latex] (h2) edge (5);
      
      \draw[blue] [-latex] (1) edge (h1);
      \draw[blue] [-latex] (3) edge (4);
    \end{tikzpicture}
}
\qquad \qquad \quad
\subfloat[]{
    \centering
    \tikzset{
      every node/.style={circle, inner sep=0.3mm, minimum size=0.45cm, draw, thick, black, fill=white, text=black},
      every path/.style={thick}
    }
    \begin{tikzpicture}[align=center]
      \node[fill=lightgray] (h1) at (-1.5,1) {$h_1$};
      \node[fill=lightgray] (h2) at (0.5,1) {$h_2$};
      \node[] (1) at (-3,0) {$v_1$};
      \node[] (2) at (-2,0) {$v_2$};
      \node[] (3) at (-1,0) {$v_3$};
      \node[] (4) at (-0,0) {$v_4$};
      \node[] (5) at (1,0) {$v_5$};
      \node[] (6) at (2,0) {$v_6$};
      
      \draw[red, dashed] [-latex] (h2) edge (h1);
      \draw[red, dashed] [-latex] (h1) edge (1);
      \draw[red, dashed] [-latex] (h1) edge (3);
      \draw[red, dashed] [-latex] (h1) edge (5);
      \draw[red, dashed] [-latex] (h2) edge (2);
      \draw[red, dashed] [-latex] (h2) edge (6);
      
      \draw[blue] [-latex] (4) edge (h1);
      \draw[blue, bend right] [-latex] (2) edge (5);
      \draw[blue] [-latex] (6) edge (5);
    \end{tikzpicture}
}
    
\caption{Two graphs that are rationally identifiable. The latent subgraphs consist only of the red dashed edges.}
\label{fig:running-example}  
\end{figure}
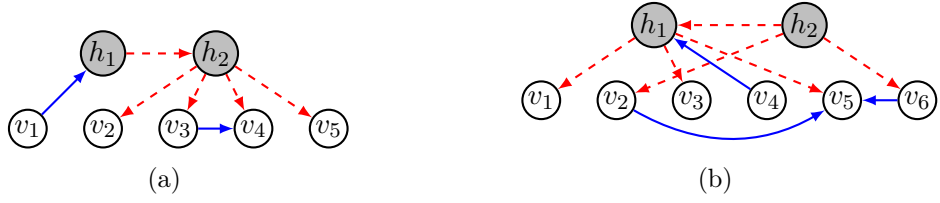

\begin{example}
In Figure~\ref{fig:running-example} we display two graphs together with their latent subgraphs. Throughout this paper, we always display graphs such that the red dashed edges correspond to the edges $D_{\lat}$ in the subgraph, while the blue edges are not in the latent subgraph, i.e., they form the complement $D \setminus D_{\lat}$.
\end{example}
 
\subsection{Identifiability}
A path from node $v$ to node $w$ in a directed graph $G$ is a sequence of edges that connects the consecutive nodes in a sequence of nodes beginning in $v$ and ending in $w$. A path from $v$ to $w$ is \emph{directed} if all its edges are pointing to $w$, that is, the path is of the form
\[
    v \rightarrow z_1 \rightarrow \cdots \rightarrow z_k \rightarrow w.
\]
For two observed nodes $v,w \in \cO$, we write $v \rightsquigarrow w \in G$ if there is a semi-direct effect from $v$ to $w$ in $G$, that is, either $v \rightarrow w \in D$ or there is a directed path from $v$ to $w$ in $G$ that only visits latent nodes. It is easy to see that $v \rightsquigarrow w \in G$ if and only if the coefficient $\olambda_{vw}$ can be nonzero. In this realm, for a node $v \in \cO$, we define $\opa(v)$  to be the set of semi-direct parents of $v$, that is, $\opa(v) = \{w \in \cO: w \rightsquigarrow v \in G\}$.

\begin{example}
In the graph in Figure~\ref{fig:running-example} (b) it holds that $v_2 \rightsquigarrow v_5, v_4 \rightsquigarrow v_5, v_6 \rightsquigarrow v_5  \in G$, and hence we have $\opa(v_5)=\{v_2,v_4,v_6\}$.
\end{example}

We are interested in the question of identifiability of all semi-direct effects $v \rightsquigarrow w \in G$, that is, whether the matrix $\oLambda$ can be uniquely recovered from a given covariance matrix $\Sigma \in \cM(G)$.  Our focus is on settings in which the identification formula is explicitly given by a rational formula, i.e., a fraction of two polynomial functions. The formula can then be readily used for downstream tasks in estimation and inference. However, since the identification formula is a rational function, there might be specific choices of parameters where the polynomial in the denominator vanishes and thus identification fails.  A subset $A \subseteq \Theta_G$ that is defined by a nonzero polynomial equation is called an \emph{algebraic set}, and it is a proper subset of the open set $\Theta_G$ \citep{cox2007ideals, shafarevich2013basic, okamoto1973distinctness}.  Hence, the set of parameter choices $A \subseteq \Theta_G$ where identifiability fails are Lebesgue measure zero sets.

\begin{definition} \label{def:rat-id}
\mbox{ }
\begin{itemize}
    \item[(a)] The graph $G$ is said to be rationally identifiable if there exist a proper algebraic subset $A \subset \Theta_G$ and a rational map $\psi: \PD(|\cO|) \rightarrow \mathbb{R}^{D_{\cO,\cO}}$ such that $\psi \circ \tau_G(\Lambda, \Phi)=\oLambda$ for all $(\Lambda, \Phi) \in \Theta_G \setminus A$.
    \item[(b)] The semi-direct effect $v \rightsquigarrow w \in G$, or also simply the coefficient $\olambda_{vw}$, is rationally identifiable if there exist a proper algebraic subset $A \subset \Theta_G$ and a rational map $\psi: \PD(|\cO|) \rightarrow \mathbb{R}$ such that $\psi \circ \tau_G(\Lambda, \Phi)=\olambda_{vw}$ for all $(\Lambda, \Phi) \in \Theta_G \setminus A$.
\end{itemize}
\end{definition}

Rational identifiability of all  semi-direct effects $v \rightsquigarrow w \in G$ is equivalent to rational identifiability of the graph $G$. We will verify in Examples~\ref{ex:id-running-example-1} and \ref{ex:id-running-example-2} that both graphs in Figure~\ref{fig:running-example} are rationally identifiable. 

\subsection{Treks} \label{sec:treks}

The covariance matrix in a linear structural equation model exhibits a nice combinatorial structure given by treks. They are obtained from ``gluing'' two directed paths together at their source node. All concepts and definitions in this section are from~\citet{sullivant2010trek}.

\begin{definition}
A \emph{trek} $\pi$ in $G=(\cO \sqcup \cL, D)$ from $v$ to $w$ is an ordered pair of directed
paths $\pi = (P_1,P_2)$ where $P_1$ has sink $v$, $P_2$ has sink $w$, and both $P_1$ and $P_2$ have the same source node $k \in \cO \sqcup \cL$.
We say that $P_1$ is the \emph{left part} of the trek, and $P_2$ is the \emph{right part} of the trek.
The common source node $k$ is called the \emph{top} of the trek, denoted $\text{top}(\pi)$. 
Note that both $P_1$ and $P_2$ may consist of a single vertex. 
\end{definition}

\begin{example} \label{ex:trek}
    In the graph in Figure~\ref{fig:running-example} (a), a trek from $v_3$ to $v_4$ is given by
    \[
        v_3 \leftarrow h_2 \leftarrow h_1 \rightarrow h_2 \rightarrow v_4.
    \]
    The top node is $h_1$, the left part is given by the directed path $h_1 \rightarrow h_2 \rightarrow v_3$, and the right part is given by $h_1 \rightarrow h_2 \rightarrow v_4$.
\end{example}

For a trek $\pi$ from $v$ to $w$ with top node $k$, we define a trek monomial as
\[
    \pi(\lambda, \phi) = \phi_k \prod_{x \rightarrow y \in \pi} \lambda_{xy}.
\]
Denote by $\mathcal{T}(v,w)$ the set of all treks from $v$ to $w$. The \emph{trek rule} \citep{spirtes2000causation, wright1934method} expresses the covariance matrix $\Sigma=(\sigma_{vw}) \in \cM(G)$ as a summation over treks, i.e.,
\begin{equation} \label{eq:trek-rule}
    \sigma_{vw} = \sum_{\pi \in \mathcal{T}(v,w)} \pi(\lambda, \phi).
\end{equation}

\begin{example}
    For the trek displayed in Example~\ref{ex:trek}, the trek monomial is given by $\pi(\lambda, \phi) = \phi_{v_1}\lambda_{v_1h_1}^2\lambda_{h_1h_2}^2\lambda_{h_2v_3}\lambda_{h_2v_4}$. By the trek rule, the corresponding entry of the covariance matrix is given by
    \begin{align*}
    \sigma_{v_3v_4} 
    = & \,\, \phi_{v_1}\lambda_{v_1h_1}^2\lambda_{h_1h_2}^2\lambda_{h_2v_3}^2\lambda_{v_3v_4} 
    + \phi_{v_1}\lambda_{v_1h_1}^2\lambda_{h_1h_2}^2\lambda_{h_2v_3}\lambda_{h_2v_4} 
    + \phi_{h_1}\lambda_{h_1h_2}^2\lambda_{h_2v_3}^2\lambda_{v_3v_4} \\
    &+ \phi_{h_1}\lambda_{h_1h_2}^2\lambda_{h_2v_3}\lambda_{h_2v_4} 
    + \phi_{h_2}\lambda_{h_2v_3}^2\lambda_{v_3v_4} 
    + \phi_{h_2}\lambda_{h_2v_3}\lambda_{h_2v_4}
    + \phi_{v_3}\lambda_{v_3v_4}.
\end{align*}
\end{example}
Now, for two sets of nodes $A,B \subseteq \cO$ of equal cardinality, there is a graphical characterization of when the determinant of the submatrix $\Sigma_{A,B}$ vanishes. It relies on trek systems with no sided intersection defined as follows.

\begin{definition}
    Consider a set of $n$ directed paths,  $\mathbf{P}=\{P_1, \ldots, P_n\}$, and let $a_i$ be the source and $b_i$ be the sink of $P_i$. If the sources are all distinct, and the sinks are all distinct, then we say that $\mathbf{P}$ is a \emph{system of directed paths} from $A=\{a_1, \ldots, a_n\}$ to $B=\{b_1, \ldots, b_n\}$. Note that there may be overlap between the sources in $A$ and the sinks in $B$, that is, we might have $A\cap B \neq \emptyset$. We say that $\mathbf{P}$ has \emph{intersection} if two paths share a node.

    Moreover, let $\Pi=\{\pi_1, \ldots, \pi_n\}$ be a set of $n$ treks, where trek $\pi_i$ starts in the node $a_i$ and ends in the node $b_i$. If the start nodes are all distinct, and the end nodes are all distinct, then we say that $\Pi$ is a \emph{system of treks} from $A=\{a_1, \ldots, a_n\}$ to $B=\{b_1, \ldots, b_n\}$. Note that $\Pi$ consists of two systems of directed paths, a \emph{left system} $\mathbf{P}_A$ from $\text{top}(\Pi)$ to $A$ and a \emph{right system} $\mathbf{P}_B$ from $\text{top}(\Pi)$ to $B$. We say that $\Pi$ has \emph{sided intersection} if  $\mathbf{P}_A$ has intersection or if $\mathbf{P}_B$ has intersection. 
\end{definition}

\begin{example}
    In the graph in Figure~\ref{fig:running-example} (b), consider the sets of nodes $A=\{v_1,v_2\}$ and $B=\{v_5,v_6\}$. The trek system 
    \[
        \{v_1 \leftarrow h_1 \rightarrow v_5, \,\,\, v_2 \leftarrow h_2 \rightarrow v_6 \}
    \]
    has no sided intersection since the paths in the left system $\{h_1 \rightarrow v_1, \,\, h_2 \rightarrow v_2\}$ don't share a node, and the paths in the right system $\{h_1 \rightarrow v_5, \,\, h_2 \rightarrow v_6\}$ don't share a node.
\end{example}

By \citet[Proposition 3.4]{sullivant2010trek}, it holds that $\det(\Sigma_{A,B})$ is generically nonzero (i.e., it is nonzero for almost all parameter choices) if and only if there is a system of treks from $A$ to $B$ with no sided intersection. Now, recall that Menger's theorem states that the maximal number of disjoint directed paths that can be found between a pair of nodes is equal to a minimum cut set. This can be translated to systems of treks with no sided intersection, which leads to the following concept.

\begin{definition} \label{def:trek-separation}
Let $A$, $B$, $C_A$, and $C_B$ be four subsets of $V=\cO \sqcup \cL$ which need not be disjoint. We say that the pair $(C_A,C_B)$ \emph{trek separates $A$ from $B$ in $G$} if for every trek $\pi=(P_1, P_2)$ in $G$ from a node in $A$ to a node in $B$, either $P_1$ contains a node in $C_A$ or $P_2$ contains a node in $C_B$. 
\end{definition}

With Definition~\ref{def:trek-separation} it holds that the determinant of $\Sigma_{A,B}$ is generically zero if and only if 
\[
    \min \{|C_A|+|C_B| : (C_A,C_B) \text{ trek separates } A \text{ from } B\} < |A| = |B|.
\]

\begin{example}
    In the graph  in Figure~\ref{fig:running-example} (b) it is impossible to trek separate the two sets $A=\{v_1,v_2\}$ and $B=\{v_5,v_6\}$ with a pair $(C_A,C_B)$ such that $|C_A|+|C_B| < 2$. We conclude that $\det(\Sigma_{A,B}) \neq 0$ for generic parameter choices. 
\end{example}
For more examples
of treks and more details on trek separation we refer to \citet{sullivant2010trek}.

\section{Latent-Subgraph Criterion} \label{sec:lsc}

In this section, we state our graphical criterion that is sufficient for rational identification of semi-direct effects. It recursively certifies identifiability of columns of $\oLambda$ and, importantly, the imposed conditions are expressed as combinatorial constraints on the given graph. Our main tool is the concept of trek separation in subgraphs that we derive in Section~\ref{sec:trek-separation}; it generalizes the standard trek separation discussed in Section~\ref{sec:treks}.
Before stating our main theorem, we define novel graphical concepts. For our criterion, we always consider one special subgraph, which is the \emph{latent subgraph} $G_{\lat}=(\cO \sqcup \cL, D_{\lat})$ with edges $D_{\lat}$ defined in Equation~\eqref{eq:latent-subgraph}.

\begin{definition}[Latent treks] \label{def:subgraph-treks}
A trek $\pi=(P_1, P_2)$  is a \emph{latent trek} if both $P_1$ and $P_2$ are directed paths in $G_{\lat}$. 
\end{definition}

The trek displayed in Example~\ref{ex:trek} is an example of a latent trek. We will also consider treks $\pi=(P_1, P_2)$ where only one of the parts $P_1$ or $P_2$ is required to be a directed path in $G_{\lat}$. We refer to such treks as treks in which the \emph{left/right part only takes edges in $G_{\lat}$}. In the next definition, we use the usual notion of \emph{descendants} of a node $v$, denoted $\des(v)$. It is given by the set of nodes $u \in \cO \sqcup \cL$ such that there is a direct path from $v$ to $u$ in $G$. Since we also allow trivial paths having no edges, we have $v \in \des(v)$.

\begin{definition}[Latent reachability]
Let $G=(\cO \sqcup \cL, D)$ be a directed graph. Let $v,w\in \cO$ be two (not necessarily distinct) observed nodes and let $H_1,H_2 \subseteq \cL$ be two sets of latent nodes. 
If there exists a latent trek $\pi=(P_1,P_2)$ from $v$ to $w$ such that $P_1$ does not contain a node in $H_1$ and $P_2$ does not contain a node in $H_2$, then we say that $w$ is \emph{latent reachable from $v$ while avoiding $(H_1,H_2)$}, and write $w \in \text{lr}_{H_1,H_2}(v)$. For a set $A \subseteq \cO$, we write $w\in\lr_{H_1,H_2}(A)$ if $w\in\lr_{H_1,H_2}(a)$ for some $a\in A$.

Moreover, we say that $w$ is \emph{extended latent reachable from $v$ while avoiding $(H_1,H_2)$} if there is a node $u \in \text{lr}_{H_1,H_2}(v)$ such that $w \in \text{des}(u)$, and write $w \in \elr_{H_1,H_2}(v)$. For a set $A \subseteq \cO$, we write $w\in\elr_{H_1,H_2}(A)$ if $w\in\elr_{H_1,H_2}(a)$ for some $a\in A$.
\end{definition}

\begin{example}
Consider the graph in Figure~\ref{fig:running-example} (b) and let $H_1=\{h_1\}$ and $H_2 = \emptyset$. The node $v_6$ is latent reachable from $v_2$ while avoiding $(H_1,H_2)$ since there is the latent trek $v_2 \leftarrow h_2 \rightarrow v_6$. It also holds that $v_5 \in \lr_{H_1,H_2}(v_2)$ since there is the latent trek $v_2 \leftarrow h_2 \rightarrow h_1 \rightarrow v_5$, and the left part $h_2 \rightarrow v_2$ does not contain the node $h_1$. On the other hand, if we consider $H_1=H_2=\{h_1\}$, then $v_5$ is not latent reachable from $v_2$ while avoiding $(H_1,H_2)$. However, it is extended latent reachable from $v_2$ while avoiding $(H_1,H_2)$ since $v_6 \in \lr_{H_1,H_2}(v_2)$ and $v_5 \in \des(v_6)$.
\end{example}

We are now ready to define our graphical criterion.

\begin{definition}[Latent-subgraph criterion] \label{def:htc}
Given a node $v \in \cO$, the $4$-tuple $(Y,Z,H_1,H_2)\in 2^{\cO \setminus \{v\}}\times 2^{\cO \setminus \{v\}}\times 2^{\cL} \times 2^{\cL}$ satisfies the \emph{latent-subgraph criterion} (LSC) with respect to $v$ if
\begin{itemize}
    \item[(i)] $|Y|=|\opa(v)|+|Z|$ and $|Z|=|H_1|+|H_2|$ with $Z \cap\opa(v) = \emptyset$, 
    \item[(ii)] $Y \cap (Z \cup \{v\}) = \emptyset$ and $(H_1,H_2)$ trek separates $Y$ and $Z \cup \{v\}$ in the latent subgraph~$G_{\lat}$, 
    \item[(iii)] there exists a system of treks with no sided intersection from $Y$ to $\opa(v) \cup Z$ in $G$ such that the left part of every trek only takes edges in $G_{\lat}$, and the right part of every trek ending in $Z$ only takes edges in $G_{\lat}$.
\end{itemize}
\end{definition}

If a tuple $(Y,Z,H_1,H_2)$ satisfies the LSC with respect to $v$, then Conditions (i) and (ii) ensure that the matrix $\Omega_{Y, Z \cup \{v\}}$ does not have full column rank. Moreover, Condition (iii) ensures, among others, that $\Omega_{Y, Z}$ generically has full column rank. This is true since there exists a system of latent treks from a subset $Y_Z \subseteq Y$ to $Z$  in the subgraph $G_{\lat}$, which implies that the determinant of $\Omega_{Y_Z, Z}$ is generically nonzero. Hence, since the ranks of $\Omega_{Y, Z \cup \{v\}}$ and $\Omega_{Y, Z}$ coincide, there must exist $\psi \in \mathbb{R}^{|Z|}$ such that $\Omega_{Y, Z} \cdot \psi = \Omega_{Y,v}$. Applying the factorization $\Sigma = (I - \oLambda)^{-\top} \Omega (I - \oLambda)^{-1}$ from~\eqref{eq:observed-covariance}, we obtain 
\[
    [(I - \oLambda)^{\top} \Sigma (I - \oLambda)]_{Y,v} - \Omega_{Y, Z} \cdot \psi = 0,
\]
which holds if and only if
\begin{equation} \label{eq:key-idea}
    \begin{pmatrix}
        [(I - \oLambda)^{\top} \Sigma]_{Y,\opa(v)} & \Omega_{Y,Z} 
    \end{pmatrix}
    \cdot
    \begin{pmatrix}
        \oLambda_{\opa(v),v} \\
        \psi
    \end{pmatrix}
    =  [(I - \oLambda)^{\top} \Sigma]_{Y,v}.
\end{equation}
Now, the entries in the matrix $[(I - \oLambda)^{\top} \Sigma]_{Y,\opa(v)}$ are given by sum over treks in which the left part of each trek only takes edges in $G_{\lat}$ and the entries in the matrix $\Omega_{Y, Z}$ are given by latent treks, in which both parts of each trek only take edges in $G_{\lat}$. Condition (iii) makes sure that the determinant of the block matrix $([(I - \oLambda)^{\top} \Sigma]_{Y,\opa(v)} \,\,\, \Omega_{Y,Z})$ is nonzero by using our criterion for trek separation in subgraphs given in Theorem~\ref{thm:generalized-trek-separation}. If we additionally make sure that suitable entries of $\oLambda$ appearing in the left-hand side of \eqref{eq:key-idea} are already known to be identifiable from earlier calculations, we can then solve for the semi-direct effects $\oLambda_{\opa(v),v}$. This is the key idea of our main result. The proof is given in Appendix~\ref{sec:proof-main-result}.

\begin{theorem}[LSC-identifiability] \label{thm:LSC}
Suppose that the $4$-tuple $(Y,Z,H_1,H_2)\in 2^{\cO \setminus \{v\}}\times 2^{\cO \setminus \{v\}}\times 2^{\cL} \times 2^{\cL}$ satisfies the LSC with respect to $v\in \cO$. If all semi-direct effects $u\rightsquigarrow w \in G$  for  $w\in Z\cup(Y\cap \elr_{H_2,H_1}(Z\cup\{v\}))$ are rationally identifiable, then all semi-direct effects $p \rightsquigarrow v \in G$ with $p \in \opa(v)$ are rationally identifiable.
\end{theorem}

With Theorem~\ref{thm:LSC} we can recursively certify rational identifiability of a graph $G$ by checking rational identifiability of the matrix $\oLambda$ column by column. If rational identifiability can be certified recursively by Theorem~\ref{thm:LSC} for all nodes $v \in \cO$, then we say that the graph $G$ is \emph{LSC-identifiable}. We provide an algorithm to check LSC-identifiability using integer linear programs in Section~\ref{sec:computation}. The proof of Theorem~\ref{thm:LSC} is constructive in the sense that is explicitly provides the rational formulas for identification of the semi-direct effects. 

\begin{example} \label{ex:id-running-example-1}
The graph in Figure~\ref{fig:running-example} (a) is LSC-identifiable, which can be verified by recursively checking all observed nodes. For each $v \in \cO$, we find a tuple $(Y,Z,H_1,H_2)$ that satisfies the LSC with respect to $v$, ensuring that all nodes in $Z \cup (Y \cap \elr_{H_2,H_1}(Z \cup \{v\}))$ were successfully checked earlier.
\begin{itemize}
    \item[] \underline{$v = v_1$:} Since $\opa(v)=\emptyset$, the tuple $(Y,Z,H_1,H_2)=(\emptyset,\emptyset,\emptyset,\emptyset)$ trivially satisfies the LSC.
    \item[] \underline{$v = v_2$:} The tuple $(Y,Z,H_1,H_2)=(\{v_1\},\emptyset,\emptyset,\emptyset)$ satisfies the LSC. Condition~(ii) holds since there is no trek from $v_1$ to $v_2$ in the latent subgraph $G_{\lat}$. The trek system for Condition~(iii) is given by the trivial trek from $v_1$ to $v_1$ that contains no edge. 
    \item[] \underline{$v = v_3, v_5$:} With the same arguments as for $v_2$ it is easy to see that the tuple $(Y,Z,H_1,H_2)=(\{v_1\},\emptyset,\emptyset,\emptyset)$ satisfies the LSC.
    \item[] \underline{$v = v_4$:} Let $(Y,Z,H_1,H_2)=(\{v_1, v_2, v_3\},\{v_5\},\emptyset,\{h_2\})$. Condition~(i) is easily checked. Condition~(ii) is satisfied since the pair $(H_1,H_2) = (\emptyset, \{h_2\})$ trek separates $(\{v_1, v_2, v_3\}$ from $\{v_4, v_5\}$ in the latent subgraph $G_{\lat}$. For Condition~(iii), the system of treks is given by $\{v_1, \,\, v_3, \,\, v_2 \leftarrow h_2 \rightarrow v_5\}$, where the single nodes $v_1$ and $v_3$ correspond to trivial treks with no edges. The single node in $Z \cup (Y \cap \elr_{H_2,H_1}(Z \cup \{v\})) = \{v_5\}$ was already successfully checked in the last step.
\end{itemize}
\end{example}

\begin{figure}[t]
    \centering
    \tikzset{
      every node/.style={circle, inner sep=0.3mm, minimum size=0.45cm, draw, thick, black, fill=white, text=black},
      every path/.style={thick}
    }
    \begin{tikzpicture}[align=center]
      \node[fill=lightgray] (h1) at (1,1) {$h_1$};
      \node[fill=lightgray] (h2) at (4,1) {$h_3$};
      \node[fill=lightgray] (h3) at (2.5,-1) {$h_2$};
      \node[] (1) at (0,0) {$v_1$};
      \node[] (2) at (1,0) {$v_2$};
      \node[] (3) at (2,0) {$v_3$};
      \node[] (4) at (3,0) {$v_4$};
      \node[] (5) at (4,0) {$v_5$};
      \node[] (6) at (5,0) {$v_6$};

      \draw[blue] [-latex] (1) edge (h1);
      \draw[blue] [-latex] (3) edge (h3);
      \draw[red, dashed] [-latex] (h1) edge (h2);
      \draw[red, dashed] [-latex] (h1) edge (2);
      \draw[red, dashed] [-latex] (h1) edge (3);
      \draw[red, dashed] [-latex] (h2) edge (4);
      \draw[red, dashed] [-latex] (h2) edge (5);
      \draw[red, dashed] [-latex] (h2) edge (6);
      \draw[red, dashed] [-latex] (h3) edge (4);
    \end{tikzpicture}
    \caption{Same graph as Figure~\ref{fig:air-pollution} with relabeled vertices.}
    \label{fig:air-pollution-relabeled}
\end{figure}
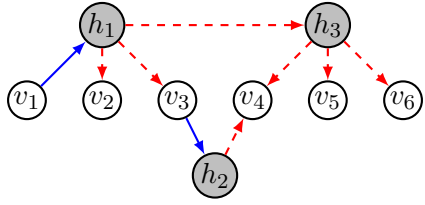

\begin{example} \label{ex:id-running-example-2}
The graphs in Figure~\ref{fig:air-pollution-relabeled} and Figure~\ref{fig:running-example} (b) are LSC-identifiable. This can be checked similarly as in Example~\ref{ex:id-running-example-1}, and for brevity we only state the tuples $(Y,Z,H_1,H_2)$ and the trek system $\Pi$ that satisfies Condition~(iii). We start with the graph in Figure~\ref{fig:air-pollution-relabeled}:
\begin{itemize}
    \item[] \underline{$v = v_1$:} $(Y,Z,H_1,H_2)=(\emptyset,\emptyset,\emptyset,\emptyset)$.
    \item[] \underline{$v = v_2,v_3,v_5,v_6$:} $(Y,Z,H_1,H_2)=(\{v_1\},\emptyset,\emptyset,\emptyset)$ and $\Pi = \{v_1\}$.
    \item[] \underline{$v = v_4$:} $(Y,Z,H_1,H_2)=(\{v_1, v_2, v_3\},\{v_5\},\emptyset,\{h_1\})$ and $\Pi = \{v_1, \,\, v_3, \,\, v_2 \leftarrow h_1 \rightarrow h_2 \rightarrow v_5\}$.
\end{itemize}
For the graph in Figure~\ref{fig:air-pollution-relabeled} (b), we have:
\begin{itemize}
    \item[] \underline{$v = v_2,v_4,v_6$:} $(Y,Z,H_1,H_2)=(\emptyset,\emptyset,\emptyset,\emptyset)$.
    \item[] \underline{$v = v_1,v_3$:} $(Y,Z,H_1,H_2)=(\{v_4\},\emptyset,\emptyset,\emptyset)$ and $\Pi = \{v_4\}$.
    \item[] \underline{$v = v_5$:} $(Y,Z,H_1,H_2)=(\{v_2, v_3, v_4, v_6\},\{v_1\},\emptyset,\{h_1\})$ and $\Pi = \{v_2, \,\, v_4, \,\, v_6, \,\, v_3 \leftarrow h_1 \rightarrow v_1\}$.
\end{itemize}
\end{example}

\begin{remark} \label{rem:comparison-htc}
    Note that the derivation of the linear equations in~\eqref{eq:key-idea} builds on fundamental insights from the half-trek criteria \citep{barber2022halftrek, foygel2012halftrek}. However, unlike their approach, we do not impose restrictions on the latent structure and, moreover, we focus on semi-direct effects that may propagate through latent variables. Consequently, our proof requires the development of entirely novel tools to certify that the matrix on the left-hand side of~\eqref{eq:key-idea} is invertible; see Section~\ref{sec:trek-separation}. It is easy to see that our criterion strictly subsumes the latent-factor half-trek criterion by \citet{barber2022halftrek} which assumes that the latent variables are restricted to be source nodes. Moreover, our algorithm to check LSC-identifiability also requires new ideas that connect trek separation in subgraphs to integer linear programming; see Section~\ref{sec:computation}. 
\end{remark}

As a special case of Theorem~\ref{thm:LSC} we obtain rational identifiability of a subclass of graphs. We say that a graph $G=(\cO \sqcup \cL, D)$ is \emph{confounding-free} if it does not contain two observed nodes $u,v \in \cO$ such that $u \rightsquigarrow v \in G$ and, in addition, there is a trek from $u$ to $v$ in the latent subgraph $G_{\lat}$.  \looseness=-1

\begin{corollary} \label{cor:bow-free}
Confounding-free acyclic graphs are rationally identifiable. 
\end{corollary}
\begin{proof}
Let $G=(\cO \sqcup \cL, D)$ be a confounding-free acyclic graph. It is easy to see that for every node $v \in \cO$ the tuple $(Y,Z,H_1,H_2)=(\opa(v), \emptyset, \emptyset, \emptyset)$ satisfies the LSC with respect to $v$. Since $G$ is acyclic, there is at least one total order $\prec$ on the observed nodes $\cO$ such that $u \rightsquigarrow v \in G$  only if $u \prec v$. Crucially, all elements $u \in \opa(v)$ are predecessors of $v$ with respect to the order $\prec$. By Theorem~\ref{thm:LSC}, we can thus certify rational identifiability of all semi-direct effects recursively according to the order $\prec $ by using the tuples $(\opa(v), \emptyset, \emptyset, \emptyset)$. 
\end{proof}

\section{Efficient Algorithm via Linear Programming} \label{sec:computation}

In this section, we propose a sound and complete algorithm for deciding whether
a graph is LSC-identifiable. 

\subsection{Path Systems in Directed Graphs} \label{sec:integer-linear-program}
We first solve the subproblem of deciding the existence of certain path systems in directed graphs that involve subgraphs. Suppose that we are given a directed graph $G=(V,D)$ and a subgraph $G_1=(V,D_1)$ with edges $D_1 \subseteq D$. Let $P,Z,Y_a \subseteq V$ be three sets of nodes such that $P$ and $Z$ are disjoint. Our goal is to decide whether there exists a subset $Y \subseteq Y_a$ such that there is a system of directed paths with no intersection from $Y$ to $Z \cup P$ where every path ending in $Z$ only has edges in $G_1$. The set $Y_a$ is a given set of nodes that are ``allowed'' as source nodes of the paths. We solve this decision problem by solving one integer linear program, which extends the usual maximum flow problem \citep{cormen2009introduction} by incorporating the additional constraint that flows ending in $A$ are restricted to only take edges the subgraph. Consider the \emph{flow graph} $G_{\flow}(Z,P,Y_a)=(V_f,D_f)$ with $V_f = V \cup \{s,t\}$ and 
\begin{align*}
    D_f=D \cup \{s \to y: y \in Y_a\} \cup \{z \to t: z \in Z\} \cup \{p \to t: p \in P\}.
\end{align*}
Moreover, let $f \in \mathbb{R}^{D_f}$ and $f^1 \in \mathbb{R}^{D_f}$ be two weight vectors for the edges in $G_{\flow}$. We define the linear program  $\text{Lp}(G,G_1,Z,P,Y_a)$ as follows:
\begin{align*}
    &\text{Maximize} \,\,\, \sum_{z \in Z} f^1_{zt} + \sum_{p \in P} f_{pt}, \\
    &\text{subject to}  \\
    &\hspace{1.6cm}\text{(i) } \,\,\, f^1_{uv} \geq 0 \text{ and } f_{uv} \geq 0 \text{ for all } u \to v \in D_f , \\[0.25cm]
    &\hspace{1.6cm}\text{(ii) } \, \sum_{u \in V_f} f^1_{uv} = \sum_{w \in V_f} f^1_{vw} \leq 1 \text{ and } \sum_{u \in V_f} f_{uv} = \sum_{w \in V_f} f_{vw} \leq 1 \text{ for all } v \in V, \\[0.25cm]
    &\hspace{1.6cm}\text{(iii) }  f^1_{uv} = 0 \text{ for all } u \to v \in D \setminus D_1 \text{ and for all } p \to t \in D_f \text{ with } p \in P, \\[0.25cm]
    &\hspace{1.6cm}\text{(iv) }  \sum_{u \in V_f} f^1_{uv} + \sum_{u \in V_f} f_{uv} \leq 1 \text{ for all } v \in V.
\end{align*}
The linear program $\text{Lp}(G,G_1,Z,P,Y_a)$ reduces to the usual maximum flow problem with node capacities equal to $1$ if the subgraph $G_1$ does not contain any edges. However, if $G_1$ contains edges, $\text{Lp}(G,G_1,Z,P,Y_a)$ is significantly different from the standard maximum flow problem. In particular, it is not equivalent to solving two maximum flow computations separately, where one maximum flow is computed on the entire graph $G$ and the other is computed on the subgraph $G_1$. This is due to Condition~(iv) that combines both flow problems.  

Note that the optimal value of $\text{Lp}(G,G_1,Z,P,Y_a)$ is bounded above by $|Z| + |P|$. We obtain the following theorem that is proved in Appendix~\ref{sec:other-proofs}.

\begin{figure}[t]
\subfloat[]{
    \centering
    \tikzset{
      every node/.style={circle, inner sep=0.3mm, minimum size=0.45cm, draw, thick, black, fill=white, text=black},
      every path/.style={thick}
    }
    \begin{tikzpicture}[align=center]
      \node[] (s1) at (-1.5,0) {$y_1$};
      \node[] (s2) at (-1.5,-1) {$y_2$};
      \node[] (a) at (1.5,0) {$z$};
      \node[] (b) at (1.5,-1) {$p$};
      \node[] (5) at (0,0) {$v$};
      \node[] (6) at (0,1) {$w$};
      
      \draw[red, dashed] [-latex] (s1) edge (5);
      \draw[red, dashed] [-latex] (s1) edge (6);
      \draw[red, dashed] [-latex] (s2) edge (5);
      \draw[red, dashed] [-latex] (5) edge (a);
      \draw[red, dashed] [-latex] (5) edge (b);

      \draw[blue] [-latex] (s2) edge (b);
      \draw[blue] [-latex] (6) edge (a);
    \end{tikzpicture}
}
\qquad \qquad \quad
\subfloat[]{
    \centering
    \tikzset{
      every node/.style={circle, inner sep=0.3mm, minimum size=0.45cm, draw, thick, black, fill=white, text=black},
      every path/.style={thick}
    }
    \begin{tikzpicture}[align=center]
        \node[] (s) at (-3,-0.5) {$s$};
      \node[] (s1) at (-1.5,0) {$y_1$};
      \node[] (s2) at (-1.5,-1) {$y_2$};
      \node[] (a) at (1.5,0) {$z$};
      \node[] (b) at (1.5,-1) {$p$};
      \node[] (5) at (0,0) {$v$};
      \node[] (6) at (0,1) {$w$};
      \node[] (t) at (3,-0.5) {$t$};
      
      \draw[red, dashed] [-latex] (s1) edge (5);
      \draw[red, dashed] [-latex] (s1) edge (6);
      \draw[red, dashed] [-latex] (s2) edge (5);
      \draw[red, dashed] [-latex] (5) edge (a);
      \draw[red, dashed] [-latex] (5) edge (b);
      \draw[red, dashed] [-latex] (a) edge (t);

      \draw[blue] [-latex] (s2) edge (b);
      \draw[blue] [-latex] (6) edge (a);
      \draw[blue] [-latex] (b) edge (t);
      \draw[blue] [-latex] (s) edge (s1);
      \draw[blue] [-latex] (s) edge (s2);
    \end{tikzpicture}
}
    \caption{(a) Directed graph where red dashed edges are elements of $D_1$, while blue edges are elements of $D \setminus D_1$. (b) Corresponding flow graph $G_{\flow}(Z,P,Y_a)$ for $Z=\{z\}$, $P=\{p\}$ and $Y_a=\{y_1,y_2\}$.}
    \label{fig:neg-example-max-flow}
\end{figure}
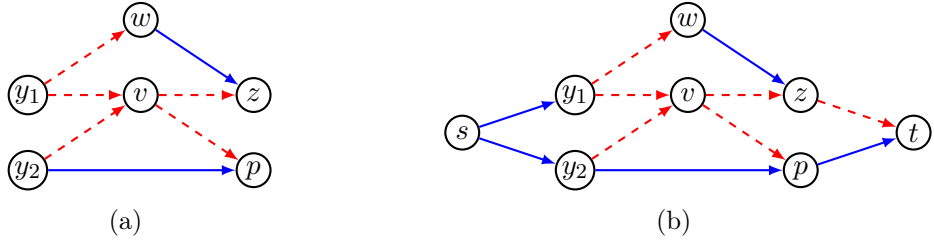

\begin{theorem} \label{thm:linear-program}
Let $G=(V,D)$ be a directed graph. 
There is an integer solution of  $\text{Lp}(G,G_1,Z,P,Y_a)$ with optimal value $|Z| + |P|$ if and only if  there exists $Y \subseteq Y_a$ such that there is a system of directed paths with no intersection from $Y$ to $Z \cup P$ where every path ending in $Z$ only has edges in $G_1$.
\end{theorem}

\begin{example}
    Consider the graph Figure~\ref{fig:neg-example-max-flow} (a) with subgraph $G_1$ given by the red dashed edges. Let $P=\{p\}$, $Z = \{z\}$ and $Y_a=\{y_1,y_2\}$.  Clearly, there exists a system of directed paths with no intersection from $Y=Y_a$ to $Z \cup P$ such that every path ending in $Z$ only has edges on $G_1$. The path system is given by $\{y_1 \rightarrow v \rightarrow z, \,\, y_2 \rightarrow p\}$. On the other hand, we obtain the flow graph $G_{\flow}(Z,P,Y_a)$ shown in Figure~\ref{fig:neg-example-max-flow} (b). Set $f^1_{s y_1} = f^1_{y_1 v}= f^1_{v z} = f^1_{z t} = 1$ and $f_{s y_2} = f_{y_2 p} = f_{pt} =1$, and set all other entries of $f$ and $f^1$ to zero. Then the tuple $(f,f^1)$ is integer-valued and maximizes $\text{Lp}(G,G_1,Z,P,Y_a)$ since $f^1_{zt} +f_{pt} = 2 = |Z|+|P|$.
\end{example}

In all our computations, whenever the optimal value of $\text{Lp}(G,G_1,Z,P,Y_a)$ was $|Z| + |P|$, it also existed an integer solution. If this is generally true, it would be sufficient to solve only the linear program, which is possible in polynomial time, instead of solving the integer linear program, which is NP-complete. 

\begin{conjecture} \label{conj:integer-solution}
If the optimal value of $\text{Lp}(G,G_1,Z,P,Y_a)$ is $|Z| + |P|$, then there is an integer solution.
\end{conjecture}

However, note that despite NP-completeness, there are very efficient methods for solving integer linear programs, such as cutting-plane methods or branch-and-bound methods; see e.g.~\citet{schrijver1986theory} or \citet{nemhauser1988integer}.

\subsection{Condition (iii) of the Latent-Subgraph Criterion}
We now consider the problem of checking the existence of the trek system in Condition~(iii) of the latent-subgraph criterion. We reduce the problem to solving the integer linear program of the previous section. Suppose we are given a set $Y_a \subseteq V$ of ``allowed nodes'' and two sets $Z, P \subseteq V$ that are disjoint. Our goal is to decide whether there exists a subset $Y \subseteq Y_a$ such that there is a system of treks with no sided intersection from $Y$ to $Z \cup P$ in $G$ where the left part of every trek only takes edges in $G_{\lat}$, and the right part of every trek ending in $Z$ only takes edges in $G_{\lat}$.

We construct a suitable graph $G^{\lp}=(V^{\lp},D^{\lp})$ and a subgraph $G^{\lp}_{\lat}=(V^{\lp},D^{\lp}_{\lat})$ that are the input to the linear program. Let $V'$ be a copy of the set of nodes $V$. Then the nodes are given by $V^{\lp} = V \cup V'$ and the set of edges are given by

\noindent\begin{minipage}{.5\linewidth}
\begin{align*}
  D^{\lp} &= \{v \rightarrow h : h \rightarrow v \in D_{\lat} \} \\
  &\,\cup \{v \rightarrow v': v \in V\} \\
  &\,\cup \{u' \rightarrow v': u \rightarrow v \in D\},
\end{align*}
\end{minipage}
\begin{minipage}{.5\linewidth}
\begin{align*}
   D^{\lp}_{\lat} &= \{v \rightarrow h : h \rightarrow v \in D_{\lat} \} \\
  &\,\cup \{v \rightarrow v': v \in V\} \\
  &\,\cup \{u' \rightarrow v': u \rightarrow v \in D_{\lat}\} .
\end{align*}
\end{minipage}
\vspace{0.3cm}

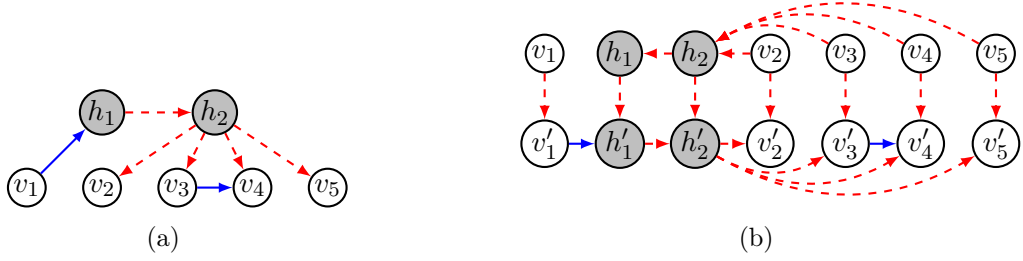
\begin{figure}[t]
\subfloat[]{
    \centering
    \tikzset{
      every node/.style={circle, inner sep=0.3mm, minimum size=0.45cm, draw, thick, black, fill=white, text=black},
      every path/.style={thick}
    }
    \begin{tikzpicture}[align=center]
      \node[fill=lightgray] (h1) at (-1,1) {$h_1$};
      \node[fill=lightgray] (h2) at (0.5,1) {$h_2$};
      
      \node[] (1) at (-2,0) {$v_1$};
      \node[] (2) at (-1,0) {$v_2$};
      \node[] (3) at (-0,0) {$v_3$};
      \node[] (4) at (1,0) {$v_4$};
      \node[] (5) at (2,0) {$v_5$};
      
      \draw[red, dashed] [-latex] (h1) edge (h2);
      \draw[red, dashed] [-latex] (h2) edge (2);
      \draw[red, dashed] [-latex] (h2) edge (3);
      \draw[red, dashed] [-latex] (h2) edge (4);
      \draw[red, dashed] [-latex] (h2) edge (5);
      
      \draw[blue] [-latex] (1) edge (h1);
      \draw[blue] [-latex] (3) edge (4);
    \end{tikzpicture}
}
\qquad \qquad \quad
\subfloat[]{
    \centering
    \tikzset{
      every node/.style={circle, inner sep=0.3mm, minimum size=0.45cm, draw, thick, black, fill=white, text=black},
      every path/.style={thick}
    }
    \begin{tikzpicture}[align=center]
      \node[fill=lightgray] (h1) at (-3,0) {$h_1$};
      \node[fill=lightgray] (h2) at (-2,0) {$h_2$};     
      \node[] (1) at (-4,0) {$v_1$};
      \node[] (2) at (-1,0) {$v_2$};
      \node[] (3) at (-0,0) {$v_3$};
      \node[] (4) at (1,0) {$v_4$};
      \node[] (5) at (2,0) {$v_5$};
      
      \node[fill=lightgray] (h1') at (-3,-1.2) {$h_1'$};
      \node[fill=lightgray] (h2') at (-2,-1.2) {$h_2'$};     
      \node[] (1') at (-4,-1.2) {$v_1'$};
      \node[] (2') at (-1,-1.2) {$v_2'$};
      \node[] (3') at (-0,-1.2) {$v_3'$};
      \node[] (4') at (1,-1.2) {$v_4'$};
      \node[] (5') at (2,-1.2) {$v_5'$};
      
      \draw[red, dashed] [-latex] (h2) edge (h1);
      \draw[red, dashed] [-latex] (2) edge (h2);
      \draw[red, dashed] [-latex, bend right] (3) edge (h2);
      \draw[red, dashed] [-latex, bend right] (4) edge (h2);
      \draw[red, dashed] [-latex, bend right] (5) edge (h2);

      \draw[red, dashed] [-latex] (h1) edge (h1');
      \draw[red, dashed] [-latex] (h2) edge (h2');
      \draw[red, dashed] [-latex] (1) edge (1');
      \draw[red, dashed] [-latex] (2) edge (2');
      \draw[red, dashed] [-latex] (3) edge (3');
      \draw[red, dashed] [-latex] (4) edge (4');
      \draw[red, dashed] [-latex] (5) edge (5');

      \draw[red, dashed] [-latex] (h1') edge (h2');
      \draw[red, dashed] [-latex] (h2') edge (2');
      \draw[red, dashed] [-latex, bend right] (h2') edge (3');
      \draw[red, dashed] [-latex, bend right] (h2') edge (4');
      \draw[red, dashed] [-latex, bend right] (h2') edge (5');
      \draw[blue] [-latex] (1') edge (h1');
      \draw[blue] [-latex] (3') edge (4');
    \end{tikzpicture}
}
\caption{(a) Directed graph. (b) The corresponding graph $G^{\lp}$ that is input to the linear program. Red dashed edges are part of the subgraph.}
\label{fig:linear-program-graph}  
\end{figure}

An example of a graph $G^{\lp}$ and its subgraph $G^{\lp}_{\lat}$ is shown in Figure~\ref{fig:linear-program-graph} (b). As a corollary of Theorem~\ref{thm:linear-program} we obtain the following result.

\begin{corollary} \label{cor:LSC-linear-program}
Let $G=(\cO \sqcup \cL,D)$ be a directed graph. There is an integer solution of  $\text{Lp}(G^{\lp},G^{\lp}_{\lat},Z,$ $P,Y_a)$ with optimal value $|Z| + |P|$ if and only if there exists $Y \subseteq Y_a$ such that there is a system of treks with no sided intersection from $Y$ to $Z \cup P$ in $G$ where the left part of every trek only takes edges in $G_{\lat}$, and the right part of every trek ending in $Z$ only takes edges in $G_{\lat}$.
\end{corollary}
\begin{proof}
Let $Y \subseteq Y_a$ be any subset. It is easy to see that there is a system of treks with no sided intersection from $Y$ to $Z \cup P$ in $G$ where the left part of every trek only takes edges in $G_{\lat}$ and the right part of every trek ending in $Z$ only takes edges in $G_{\lat}$ if and only if there is a path system with no intersection from $Y$ to $Z \cup P$ in $G^{\lp}$ where where every path ending in $Z$ only has edges in $G^{\lp}_{\lat}$. \looseness=-1
\end{proof}

\subsection{LSC-Identifiability}
Next, we give a recursive algorithm to decide whether a graph is LSC-identifiable. For each observed node $v \in \cO$, we iterate over suitable sets of nodes $H_1,H_2 \subseteq \cL$ and $Z\in\cO$ and then apply Corollary~\ref{cor:LSC-linear-program} to find a tuple $(Y,Z,H_1,H_2)$ that satisfies the latent-subgraph criterion with respect to $v$. By restricting the set of ``allowed nodes'' for $Z$ and for $Y$, we make sure that all conditions of the latent-subgraph criterion are satisfied.  The next two lemmas present requirements for the nodes in $Z$ and in $Y$.

\begin{algorithm}[t]
\caption{Deciding LSC-identifiability.}
\begin{algorithmic}[1]
\REQUIRE Graph $G = (\cO \sqcup \cL, D)$.\\
\ENSURE Solved nodes $S = \{v \in \cO: \opa(v) = \emptyset\}$.
\REPEAT 
    \FOR {$v \in V \setminus S$}
        \FOR {$H_1, H_2 \subseteq \mathcal{L}$}
            \STATE {
                $Z_a =  \{w \in S \setminus (\{v\} \cup \opa(v)): \text{there is a latent trek from a node in } H_1 \text{ to } w \text{ or }$\\
                \hspace{4.91cm} $\text{there is a directed path in } G_{\lat} \text{ from a node in } H_2 \text{ to } w \}.$
            }
            \FOR {$Z \subseteq Z_a$ such that $|Z|=|H_1| + |H_2|$}
                \STATE {$Y_a^c =  \{w \in \cO: w \in \elr_{H_2,H_1}(Z \cup \{v\}) \setminus S \text{ or } w \in \text{lr}_{H_2,H_1}(Z \cup \{v\})\}$ and $Y_a = \cO \setminus Y_a^c$.}
                \IF {there is an integer solution of  $\text{Lp}(G^{\lp},G^{\lp}_{\lat},Z,\opa(v),Y_a)$  with optimal value $|Z| + |\opa(v)|$}
                    \STATE {$S = S \cup \{v\}$.}
                    \BREAK { all for-loops.}
                \ENDIF
            \ENDFOR
        \ENDFOR
    \ENDFOR
\UNTIL{$S = \cO$ or no change has occurred in the last iteration.}
\RETURN ``yes'' if $S=\cO$, ``no'' otherwise.
\end{algorithmic}
\label{alg:check-LSC-id}
\end{algorithm}

\begin{lemma} \label{lem:allowed-for-Z}
Suppose that $(Y, Z, H_1, H_2) \in 2^{\cO\setminus\{v\}}\times 2^{\cO\setminus \{v\}}\times 2^\cL\times 2^\cL$ is a tuple satisfying the LSC with respect to $v$. Then, for each node $z \in Z$, there is a latent trek from a node in $H_1$ to $z$ or there is a directed path in $G_{\lat}$ from a node in $H_2$ to $z$.
\end{lemma}
\begin{proof}
By Condition~(iii) of the LSC there is a subset $Y_Z \subseteq Z$ with $|Y_Z|=|Z|$ such that there is a system of latent treks $\Pi$ from $Y_Z$ to $Z$. Recall that a latent trek is a trek in the subgraph $G_{\lat}$. On the other hand, by Condition~(ii) of the LSC the pair $(H_1, H_2)$ trek separates $Y_Z$ and $Z \cup \{v\}$ in the subgraph $G_{\lat}$. Hence, for any trek $\pi \in \Pi$, either the left part contains a node in $H_1$ or the right part contains a node in $H_2$, which concludes the proof.
\end{proof}

\begin{lemma} \label{lem-allowed-for-Y}
Suppose that $(Y, Z, H_1, H_2) \in 2^{\cO\setminus\{v\}}\times 2^{\cO\setminus \{v\}}\times 2^\cL\times 2^\cL$ is a tuple satisfying the LSC with respect to $v$. Then, each node $y \in Y$ is not in $\lr_{H_2,H_1}(Z \cup \{v\})$.
\end{lemma}
\begin{proof}
Consider a node $y \in Y$ and suppose that $y \in \lr_{H_2,H_1}(Z \cup \{v\})$. Then, the pair $(H_1, H_2)$ does not trek separate $y$ and $Z \cup \{v\}$, which is a contradiction to Condition~(ii) of the LSC. 
\end{proof}

When searching for tuples that satisfy the latent-subgraph criterion with respect to $v$, we also have to make sure that, for a possible solution $(Y,Z,H_1,H_2)$, each node $w\in Z\cup(Y\cap \elr_{H_2,H_1}(Z\cup\{v\}))$ was solved before. Our procedure to decide LSC-identifiability is formalized in Algorithm~\ref{alg:check-LSC-id}. In each iteration, the sets $Z_a$ and $Y_a$ denote the sets of “allowed nodes” for $Z$ and $Y$, respectively. We show in the next theorem that the algorithm is sound and complete for deciding LSC-identifiability. The proof is given in Appendix~\ref{sec:other-proofs}.

\begin{theorem} \label{thm:sound-and-complete}
A graph $G=(\cO \sqcup \cL, D)$ is LSC-identifiable if and only if Algorithm \ref{alg:check-LSC-id} returns ``yes''.
\end{theorem}

\begin{remark}
If we only allow sets $H_1, H_2$ with $|H_1|+|H_2| \leq k$ in line 3 for fixed $k \in \mathbb{N}$, then the algorithm solves the integer linear program at most $\cO^{2+k} \cL^{2k}$ times. Therefore, with this restriction, the algorithm is polynomial time if Conjecture~\ref{conj:integer-solution} is true. On the other hand, without bounding $|H_1|+|H_2|$ deciding LSC-identifiability is NP-hard. This can be seen by the fact that deciding whether a graph satisfies the latent-factor half-trek criterion is NP-hard \citep{barber2022halftrek}, and recalling that the latent-subgraph criterion strictly subsumes the latent-factor half-trek criterion (Remark~\ref{rem:comparison-htc}).
\end{remark}

\subsection{Numerical Experiments} \label{sec:numerical-experiments}
We conduct a small simulation study to demonstrate the practical applicability of Algorithm~\ref{alg:check-LSC-id}. An implementation of Algorithm~\ref{alg:check-LSC-id} and code to reproduce the experiments is available at \url{https://github.com/NilsSturma/LSC}. We randomly generate graphs on $|\cO|=10$ observed nodes and $|\cL|=5$ latent nodes.  For different edge probabilities $p$, we generate adjacency matrices $A \in \mathbb{R}^{15 \times 15}$ of directed acyclic graphs by independently sampling $A_{ij} \sim \text{Ber}(p)$ if $i < j$ and setting $A_{ij}=0$ else. The graph $G=(V,D)$ on the nodes $V=\{1, \ldots, 15\}$ is then obtained by randomly splitting $V$ into a set with $10$ observed nodes and a set with $5$ latent nodes, and by setting $D=\{i \rightarrow j: A_{ij}=1\}$. When checking LSC-identifiability with Algorithm~\ref{alg:check-LSC-id}, we bound $|H_1|+|H_2|\leq k$ in line $3$ for $k=1,2,3$.

\begin{table*}[t]\centering 
\begin{tabular}{@{}c|rrr@{}}\toprule
edge prob. & $k=1$ & $k=2$ & $k=3$ \\ 
\midrule
0.15 & 849 & 850 & 850 \\
0.20 & 727 & 734 & 734 \\
0.25 & 599 & 607 & 608 \\
0.30 & 423 & 441 & 442 \\
0.35 & 270 & 287 & 287 \\
0.40 & 164 & 180 & 180 \\
0.45 & 103 & 109 & 109 \\
\midrule
Total & 3135 & 3208 & 3210 \\
\bottomrule
\end{tabular}
\caption{Counts of graphs that are certified to be LSC-identifiable by Algorithm~\ref{alg:check-LSC-id} with the bound $|H_1|+|H_2|\leq k$. For each edge probability a total number of $1000$ graphs was randomly generated.}
\label{table:10-5}
\end{table*}

Table~\ref{table:10-5} lists counts of how many out of $1000$ randomly generated graphs for each edge probability are LSC-identifiable by applying Algorithm~\ref{alg:check-LSC-id} with different bounds $k$ on $|H_1|+|H_2|$. As expected, the denser the graph, the less graphs are certified to be LSC-identifiable. We also see that the gap between the columns $k=2$ and $k=3$ is very small. That is, $|H_1|+|H_2|\leq 2$ is in most cases enough to certify LSC-identifiability, and larger sets $H_1$ and $H_2$  are rarely needed.

\section{The Canonical Model} \label{sec:canonical}
As mentioned in the introduction, the dominant approach in state-of-the-art methods to handle settings with explicitly modeled latent variables is to transform the models into \emph{canonical models} that correspond to graphs where each latent node is a source node. The procedure is, for example, described by \citet{hoyer2008estimation}; we summarize it in the following definition.
\begin{definition}
    Let $G=(\cO \sqcup \cL,D)$ be a graph. The \emph{canonical graph} $G_{\can}=(\cO \sqcup \cL,D_{\can})$ has the same set of nodes $\cO \sqcup \cL$, and the edges $D_{\can}$ are given as follows:
    \begin{itemize}
        \item[(i)] For $v,w \in \cO$, we have that $v \rightarrow w \in D_{\can}$ whenever $v \rightsquigarrow w \in G$.
        \item[(ii)] For $v \in \cL$ and $w \in \cO$,  we have that $v \rightarrow w \in D_{\can}$ whenever there is a directed path from $v$ to $w$ in $G_{\lat}$.
    \end{itemize}
    We also say that the graph $G_{\can}=(\cO \sqcup \cL,D_{\can})$ is the \emph{canonicalization} of $G$.
\end{definition}

\begin{figure}[t]
\subfloat[]{
    \centering
    \tikzset{
      every node/.style={circle, inner sep=0.3mm, minimum size=0.45cm, draw, thick, black, fill=white, text=black},
      every path/.style={thick}
    }
    \begin{tikzpicture}[align=center]
      \node[fill=lightgray] (h1) at (-1.5,1) {$h_1$};
      \node[fill=lightgray] (h2) at (0.5,1) {$h_2$};
      \node[fill=lightgray] (h3) at (-0.5,-1) {$h_3$};
      \node[] (1) at (-2,0) {$v_1$};
      \node[] (2) at (-1,0) {$v_2$};
      \node[] (3) at (-0,0) {$v_3$};
      \node[] (4) at (1,0) {$v_4$};
      \node[] (5) at (2,0) {$v_5$};
      
      \draw[red, dashed] [-latex] (h1) edge (1);
      \draw[red, dashed] [-latex] (h1) edge (2);
      \draw[red, dashed] [-latex] (h1) edge (h2);
      \draw[red, dashed] [-latex] (h2) edge (2);
      \draw[red, dashed] [-latex] (h2) edge (3);
      \draw[red, dashed] [-latex] (h2) edge (4);
      \draw[red, dashed] [-latex] (h2) edge (5);
      \draw[red, dashed] [-latex] (h3) edge (3);
      \draw[red, dashed] [-latex] (h3) edge (4);
      
      \draw[blue] [-latex] (1) edge (h3);
    \end{tikzpicture}
    \qquad \qquad \qquad
   \begin{tikzpicture}[align=center]
      \node[fill=lightgray] (h1) at (-1.5,1) {$h_1$};
      \node[fill=lightgray] (h2) at (0.5,1) {$h_2$};
      \node[fill=lightgray] (h3) at (0.5,-1) {$h_3$};
      \node[] (1) at (-2,0) {$v_1$};
      \node[] (2) at (-1,0) {$v_2$};
      \node[] (3) at (-0,0) {$v_3$};
      \node[] (4) at (1,0) {$v_4$};
      \node[] (5) at (2,0) {$v_5$};
      
      \draw[red, dashed] [-latex] (h1) edge (1);
      \draw[red, dashed] [-latex] (h1) edge (2);
      \draw[red, dashed] [-latex] (h1) edge (3);
      \draw[red, dashed] [-latex] (h1) edge (4);
      \draw[red, dashed] [-latex] (h1) edge (5);
      \draw[red, dashed] [-latex] (h2) edge (2);
      \draw[red, dashed] [-latex] (h2) edge (3);
      \draw[red, dashed] [-latex] (h2) edge (4);
      \draw[red, dashed] [-latex] (h2) edge (5);
      \draw[red, dashed] [-latex] (h3) edge (3);
      \draw[red, dashed] [-latex] (h3) edge (4);
      
      \draw[blue] [-latex, bend right] (1) edge (3);
      \draw[blue] [-latex, in=-140, out=-40] (1) edge (4);
    \end{tikzpicture}
    }

\subfloat[]{
    \centering
    \tikzset{
      every node/.style={circle, inner sep=0.3mm, minimum size=0.45cm, draw, thick, black, fill=white, text=black},
      every path/.style={thick}
    }
    \begin{tikzpicture}[align=center]
      \node[fill=lightgray] (h1) at (-1.875,1) {$h_1$};
      \node[fill=lightgray] (h2) at (-0.625,1) {$h_2$};
      \node[fill=lightgray] (h3) at (0.625,1) {$h_3$};
      \node[] (1) at (-2.5,0) {$v_1$};
      \node[] (2) at (-1.25,0) {$v_2$};
      \node[] (3) at (-0,0) {$v_3$};
      \node[] (4) at (1.25,0) {$v_4$};
      \node[] (5) at (2.5,0) {$v_5$};
      \node[] (6) at (3.75,0) {$v_6$};
      
      \draw[red, dashed] [-latex] (h1) edge (1);
      \draw[red, dashed] [-latex] (h1) edge (h2);
      \draw[red, dashed] [-latex] (h2) edge (2);
      \draw[red, dashed] [-latex] (h2) edge (h3);
      \draw[red, dashed] [-latex] (h3) edge (3);
      \draw[red, dashed] [-latex] (h3) edge (4);
      \draw[red, dashed] [-latex] (h3) edge (5);
      \draw[red, dashed] [-latex] (h3) edge (6);
      
      \draw[blue] [-latex] (5) edge (6);
    \end{tikzpicture}
    \qquad \qquad \qquad
   \begin{tikzpicture}[align=center]
      \node[fill=lightgray] (h1) at (-1.875,1) {$h_1$};
      \node[fill=lightgray] (h2) at (-0.625,-1) {$h_2$};
      \node[fill=lightgray] (h3) at (0.625,1) {$h_3$};
      \node[] (1) at (-2.5,0) {$v_1$};
      \node[] (2) at (-1.25,0) {$v_2$};
      \node[] (3) at (-0,0) {$v_3$};
      \node[] (4) at (1.25,0) {$v_4$};
      \node[] (5) at (2.5,0) {$v_5$};
      \node[] (6) at (3.75,0) {$v_6$};
      
      \draw[red, dashed] [-latex] (h1) edge (1);
      \draw[red, dashed] [-latex] (h1) edge (2);
      \draw[red, dashed] [-latex] (h1) edge (3);
      \draw[red, dashed] [-latex] (h1) edge (4);
      \draw[red, dashed] [-latex] (h1) edge (5);
      \draw[red, dashed] [-latex] (h1) edge (6);
      \draw[red, dashed] [-latex] (h2) edge (2);
      \draw[red, dashed] [-latex] (h2) edge (3);
      \draw[red, dashed] [-latex] (h2) edge (4);
      \draw[red, dashed] [-latex] (h2) edge (5);
      \draw[red, dashed] [-latex] (h2) edge (6);
      \draw[red, dashed] [-latex] (h3) edge (3);
      \draw[red, dashed] [-latex] (h3) edge (4);
      \draw[red, dashed] [-latex] (h3) edge (5);
      \draw[red, dashed] [-latex] (h3) edge (6);
      
      \draw[blue] [-latex] (5) edge (6);
    \end{tikzpicture}
    }

        \caption{Two graphs and their canonicalizations. Left graphs are rationally identifiable, while right graphs are not.}
    \label{fig:example-canonical-1}  
\end{figure}
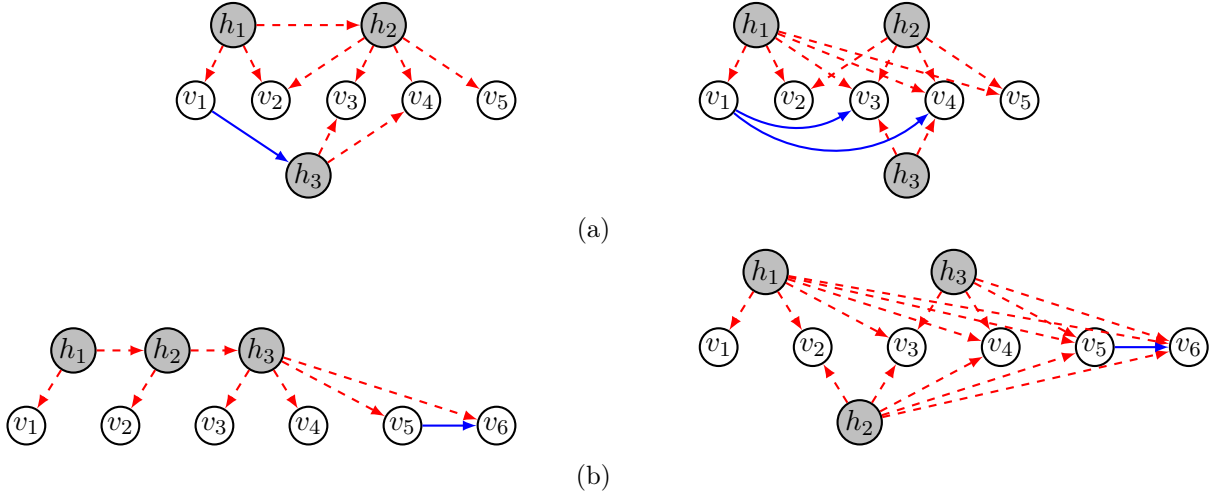

Figure~\ref{fig:example-canonical-1} shows two graphs together with their canonicalizations. In canonical graphs, every latent node is a source node. It is easy to see that the model corresponding to a given graph $G$ is always a submodel of the model defined by its canonicalization $G_{\can}$. This is formally shown in the next lemma.
\begin{lemma} \label{lem:can-submodel}
    Let $G=(\cO \sqcup \cL,D)$ be a graph. Then, it holds that $\cM(G) \subseteq \cM(G_{\can})$.
\end{lemma}

\begin{proof}
Let $G=(\cO \sqcup \cL,D)$ be a graph and consider a tuple $(\Lambda, \Phi) \in \Theta_G$ in the domain of the parameterization $\tau_G$. It is enough to construct parameters $(\Lambda', \Phi') \in \Theta_{G_{\can}}$ such that $\tau_{G_{\can}}(\Lambda', \Phi') = \tau_G(\Lambda, \Phi)$. Define
\begin{equation}
\begin{aligned} \label{eq:can-parameters}
    \Lambda'_{\cO,\cO} &= \Lambda_{\cO,\cO} + \Lambda_{\cO,\cL}(I-\Lambda_{\cL,\cL})^{-1} \Lambda_{\cL,\cO}, \\
    \Lambda'_{\cL,\cO} &= (I-\Lambda_{\cL,\cL})^{-1} \Lambda_{\cL,\cO},
\end{aligned}
\end{equation}
and set $ \Lambda'_{\cL,\cL} = 0$, $ \lambda'_{\cO,\cL} = 0$ and $\Phi' = \Phi$. By definition of the canonicalization, we see that $\lambda'_{vw} = 0$ whenever $v \rightarrow w \not\in D_{\can}$. Moreover, since $\Lambda \in \mathbb{R}_{\reg}^D$, it holds that the matrices $I - \oLambda' = I - \Lambda'_{\cO,\cO} = I - \oLambda$ and $I -  \Lambda'_{\cL,\cL} = I$ are both invertible, and thus $\Lambda'  \in \mathbb{R}_{\reg}^{D_{\can}}$. Therefore, we have shown that $(\Lambda', \Phi') \in \Theta_{G_{\can}}$. To finish the proof, note that
\begin{align*}
    \Omega' &= {\Lambda'}_{\cL,\cO}^{\top}  (I -\Lambda'_{\cL,\cL})^{-\top} \Phi'_{\cL,\cL} (I -\Lambda'_{\cL,\cL})^{-1} \Lambda'_{\cL,\cO} + \Phi'_{\cO,\cO} \\
    &= \Lambda_{\cL,\cO}^{\top}  (I -\Lambda_{\cL,\cL})^{-\top} \Phi_{\cL,\cL} (I -\Lambda_{\cL,\cL})^{-1} \Lambda_{\cL,\cO} + \Phi_{\cO,\cO} = \Omega,
\end{align*}
which implies that
\[
    \tau_{G_{\can}}(\Lambda', \Phi') = (I - \oLambda')^{-\top} \Omega' (I - \oLambda')^{-1} =  (I - \oLambda)^{-\top} \Omega (I - \oLambda)^{-1} = \tau_G(\Lambda, \Phi).
\]
\end{proof}

It is crucial to note that the inclusion in Lemma~\ref{lem:can-submodel} may be strict, and even the dimensions of both models may be different. On the one hand, it may occur that the dimension of the model $\cM(G_{\can})$ is strictly larger than the dimension of the set of tuples $(\oLambda, \Omega)$, while the dimension of $\cM(G)$ and the dimension of the set of tuples  $(\oLambda, \Omega)$ coincide. In this case, $\cM(G_{\can})$ can not be rationally identifiable, while the ``original'' model $\cM(G)$ may be rationally identifiable.

\begin{figure}[t]    
    \centering
    \tikzset{
      every node/.style={circle, inner sep=0.3mm, minimum size=0.45cm, draw, thick, black, fill=white, text=black},
      every path/.style={thick}
    }
    \begin{tikzpicture}[align=center]
      \node[fill=lightgray] (h1) at (0,0) {$h_1$};
      \node[fill=lightgray] (h2) at (2.3,1.3) {$h_2$};
      \node[fill=lightgray] (h3) at (2.3,-1.3) {$h_3$};
      \node[] (1) at (-1,1) {$v_1$};
      \node[] (2) at (-1,-1) {$v_2$};
      \node[] (3) at (1,1) {$v_3$};
      \node[] (4) at (1,-1) {$v_4$};
      \node[] (5) at (2,0) {$v_5$};

      \draw[blue] [-latex] (1) edge (h1);
      \draw[blue] [-latex] (2) edge (h1);
      \draw[red, dashed] [-latex] (h1) edge (3);
      \draw[red, dashed] [-latex] (h1) edge (4);
      \draw[blue] [-latex] (3) edge (5);
      \draw[blue] [-latex] (4) edge (5);
      \draw[red, dashed] [-latex] (h2) edge (3);
      \draw[red, dashed] [-latex] (h2) edge (5);
      \draw[red, dashed] [-latex] (h3) edge (4);
      \draw[red, dashed] [-latex] (h3) edge (5);
    \end{tikzpicture}
    \qquad \qquad \qquad
    \begin{tikzpicture}[align=center]
      \node[fill=lightgray] (h1) at (1,0) {$h_1$};
      \node[fill=lightgray] (h2) at (2.3,1.3) {$h_2$};
      \node[fill=lightgray] (h3) at (2.3,-1.3) {$h_3$};
      \node[] (1) at (-1,1) {$v_1$};
      \node[] (2) at (-1,-1) {$v_2$};
      \node[] (3) at (1,1) {$v_3$};
      \node[] (4) at (1,-1) {$v_4$};
      \node[] (5) at (2,0) {$v_5$};
      
      \draw[blue] [-latex] (1) edge (3);
      \draw[blue] [-latex] (1) edge (4);
      \draw[blue] [-latex] (2) edge (3);
      \draw[blue] [-latex] (2) edge (4);
      \draw[red, dashed] [-latex] (h1) edge (3);
      \draw[red, dashed] [-latex] (h1) edge (4);
      \draw[blue] [-latex] (3) edge (5);
      \draw[blue] [-latex] (4) edge (5);
      \draw[red, dashed] [-latex] (h2) edge (3);
      \draw[red, dashed] [-latex] (h2) edge (5);
      \draw[red, dashed] [-latex] (h3) edge (4);
      \draw[red, dashed] [-latex] (h3) edge (5);
    \end{tikzpicture}

        \caption{Graph with its canonicalization. Right graph is rationally identifiable, while left graph is not.}
    \label{fig:example-canonical-2}  
\end{figure}
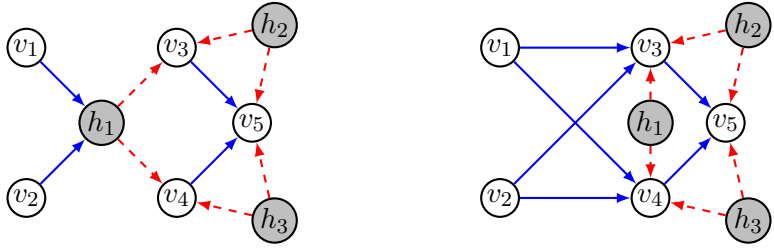

\begin{example}
    Consider the left graph in Figure~\ref{fig:example-canonical-1} (a). The semi-direct effects $v_1 \rightsquigarrow v_3$ and $v_1 \rightsquigarrow v_4$ are rationally identifiable by our new criterion given in Theorem~\ref{thm:LSC}. For example, for $v_1 \rightsquigarrow v_3$, the tuple $(Y,Z,H_1,H_2)=(\{v_1,v_2\},\{v_5\},\emptyset,\{h_2\})$ satisfies the LSC. On the other hand, the canonicalization is not rationally identifiable as we checked with techniques from computer algebra \citep{garcia2010identifying}. One way to see this is that the dimension of the set of matrices $\Omega=\Lambda_{\cL,\cO}^{\top} \Phi_{\cL,\cL} \Lambda_{\cL,\cO} + \Phi_{\cO,\cO}$ for $\Lambda \in \mathbb{R}^{D_{\can}}$ is $10$. The dimension of $\PD(4)$ is also equal to $10$, which implies that $\dim(\cM(G)) \leq 10$. Hence, the mapping $(\oLambda, \Omega) \mapsto (I-\oLambda)^{-\top} \Omega (I-\oLambda)^{-1}$ can not be one-to-one.

    Another example is the left graph in Figure~\ref{fig:example-canonical-1} (b). The effect $v_5 \rightsquigarrow v_6$ is identifiable by taking $(Y,Z,H_1,H_2)=(\{v_4,v_5\},\{v_3\}, \{h_3\}, \emptyset)$. But the canonicalization is not rationally identifiable since the dimension of the set of matrices $\Omega$ is $21$, which is already equal to the dimension of $\PD(6)$, the ambient space of $\cM(G_{\can})$.
\end{example}

Surprisingly, the other direction does also not hold. A canonical graph $G_{\can}$ being rationally identifiable does \emph{not} imply that all graphs $G$ with the same canonicalization $G_{\can}$ are rationally identifiable. Recall from Definition~\ref{def:rat-id} that rational identifiability usually only implies that we can recover $\oLambda'$ for all parameter choices $(\Lambda', \Phi') \in \Theta_{G_{\can}}$ that are not part of a lower dimensional subset $A \subset \Theta_{G_{\can}}$. However, the set of parameters we obtain from the ``original'' model via the formulas in~\eqref{eq:can-parameters}, may also only form a lower dimensional subset of $\Theta_{G_{\can}}$. If this subset is included in $A$, rational identifiability may fail. Figure~\ref{fig:example-canonical-2} provides an example  where $G_{\can}$ is rationally identifiable and $G$ is not.

Crucially, this example shows that it is not sufficient to study identifiability only for models given by canonical graphs in which all latent nodes are source nodes. If a canonical graph is rationally identifiable, identifiability may rest on the assumption that direct causal effects are not mediated by the same latent variables. 

\begin{table*}[t]\centering 
\begin{tabular}{@{}c|rrrr@{}}\toprule
edge prob. & $G$ & $G_{\can}$ & $G \setminus G_{\can}$ & $G_{\can} \setminus G$\\ 
\midrule
0.15 & 850 & 840 & 10 & 0 \\
0.20 & 734 & 725 & 12 & 3 \\
0.25 & 608 & 569 & 39 & 0 \\
0.30 & 442 & 409 & 34 & 1 \\
0.35 & 287 & 264 & 23 & 0 \\
0.40 & 180 & 164 & 17 & 1 \\
0.45 & 109 & 103 & 6 & 0 \\
\midrule
Total & 3210 & 3074 & 141 & 5 \\
\bottomrule
\end{tabular}
\caption{Counts of graphs and their canonicalizations that are certified to be LSC-identifiable by Algorithm~\ref{alg:check-LSC-id} with the bound $|H_1|+|H_2|\leq 3$. For each edge probability a total number of $1000$ graphs was randomly generated. The last two columns list counts of set differences, e.g, in  column ``$G \setminus G_{\can}$'' we count graphs that are LSC-identifiable but their canonicalization is not.}
\label{table:10-5-can}
\end{table*}

\begin{example}
Recall the results of our numerical experiments in Section~\ref{sec:numerical-experiments}. We also computed the canonicalization of each generated graph and checked whether it is LSC-identifiable using Algorithm~\ref{alg:check-LSC-id}. Table~\ref{table:10-5-can} compares the counts of LSC-identifiable graphs with the counts of their LSC-identifiable canonicalizations. As expected, we see in the last two columns that both cases occur: There are graphs that are LSC-identifiable but their canonicalization is not and vice versa.
\end{example}

\section{Trek Separation in Subgraphs} \label{sec:trek-separation}

In this section, we derive a graphical criterion to verify that the determinant of block matrices as in the left-hand side of~\eqref{eq:key-idea} is nonzero. Let $G=(V,D)$ be a directed graph and suppose that $G_1=(V,D_1)$ and $G_2=(V,D_2)$ are two subgraphs on the same set of nodes with $D_1 \subseteq D$ and $D_2 \subseteq D$. Let $\lambda=\{\lambda_{vw}: v,w \in V\}$ and $\phi=\{\phi_v: v \in V\}$ be two collections of indeterminates.
As before, we define $\Lambda$ to be sparse $|V| \times |V|$ matrices of unknowns, where the $vw$-th entry is given by $\lambda_{vw}$ if $v \rightarrow w \in D$, and zero otherwise. 
Similarly, define $\Lambda_1$ and $\Lambda_2$ to be the sparse $|V| \times |V|$ matrix of unknowns, where the $vw$-th entry of $\Lambda_1$ is given by $\lambda_{vw}$ if $v \rightarrow w$ in $D_1$, and zero otherwise, and the $vw$-th entry of $\Lambda_2$ is given by $\lambda_{vw}$ if $v \rightarrow w$ in $D_2$, and zero otherwise. We assume that $\det(I-\Lambda)$, $\det(I-\Lambda_1)$ and $\det(I-\Lambda_2)$ are not equal to zero. Note that the statement that a determinant is nonzero in this section means that the determinant is not the zero polynomial or power series. Finally, let $\Phi$ be a diagonal $|V| \times |V|$ matrix of unknowns, where the diagonal entry $\Phi_{vv}$ is given by the unknown $\phi_v$.

Now, consider four sets $A,B,C,D \subseteq V$ such that $A$ and $B$ are disjoint and $C$ and $D$ are disjoint and that $|A \cup B| = |C \cup D|$. Define the matrix 
\begin{align}
    M &= \begin{blockarray}{ccc}
    & C & D \\
    \begin{block}{c(cc)}
       A & (I-\Lambda_1)^{-\top} \Phi (I-\Lambda)^{-1} & (I-\Lambda_1)^{-\top} \Phi (I-\Lambda_2)^{-1} \\
       B & (I-\Lambda)^{-\top} \Phi (I-\Lambda)^{-1} & (I-\Lambda)^{-\top} \Phi (I-\Lambda_2)^{-1} \nonumber \\
    \end{block}
    \end{blockarray} \\
    &= \underbrace{\begin{pmatrix}
        [(I-\Lambda_1)^{-\top}]_{A,V} \\
        [(I-\Lambda)^{-\top}]_{B,V}
    \end{pmatrix}}_{=:(L_{V, A \cup B})^{\top}}
    \cdot \Phi \cdot 
    \underbrace{\begin{pmatrix}
        [(I-\Lambda)^{-1}]_{V,C} & [(I-\Lambda_2)^{-1}]_{V,D}
    \end{pmatrix}}_{=:R_{V, C \cup D}}. \label{eq:M-matrix}
\end{align}
We are interested in a graphical criterion of when the determinant of $M$ is nonzero. From the Cauchy-Binet determinant expansion formula, one directly obtains the following lemma. The exact proof is equivalent to the proof of Lemma 3.2 in \citet{sullivant2010trek}.

\begin{lemma} \label{lem:cauchy-binet}
    The determinant $\det(M)$ is identically zero if and only if for every set $S \subseteq V$ with $|S|=|A \cup B|=|C \cup D|$ either $\det(L_{S, A \cup B})=0$ or $\det(R_{S, C \cup D})=0$.
\end{lemma}

Hence, to understand when $\det(M)$ is nonzero, we need to understand when $\det(L_{S, A \cup B})$ is nonzero. For a directed path $P$, we define the path monomial as $P(\lambda)=\prod_{v \rightarrow w \in P} \lambda_{vw}$, and for a system of paths $\mathbf{P}=\{P_1, \ldots, P_n\}$ we define the monomial $\bfP(\lambda)=\prod_{i=1}^n P_i$. The proof of the next lemma is given in Appendix~\ref{sec:other-proofs}.

\begin{lemma} \label{lem:equal-path-systems}
    Let $S \subseteq V$ be a set of nodes such that $|S|=|A \cup B|$. Let $\mathbf{P}$ be a system of directed paths in $G$ with no intersection from $S$ to $A \cup B$ such that no path contains a cycle. If for another system of directed paths $\Psi$ from $S$ to $A \cup B$ (possibly with intersection) the monomial $\mathbf{P}(\lambda)=\Psi(\lambda)$, then $\mathbf{P}=\Psi$.
\end{lemma}

We obtain a sufficient condition that certifies when $\det(L_{S, A \cup B})$ is nonzero; it is also proved in Appendix~\ref{sec:other-proofs}.

\begin{lemma} \label{lem:det-path-nonzero}
     Let $S \subseteq V$ be a set of nodes such that $|S|=|A \cup B|$. If there is a system of directed paths with no intersection from $S$ to $A \cup B$ such that every path ending in $A$ only has edges in $G_1$, then $\det(L_{S, A \cup B})\neq0$.
\end{lemma}

Combining Lemma~\ref{lem:det-path-nonzero} with Lemma~\ref{lem:cauchy-binet} allows us to present a sufficient condition for when the determinant of $M$ is nonzero.

\begin{theorem} \label{thm:generalized-trek-separation}
If there is a system of treks with no sided intersection from $A \cup B$ to $C \cup D$ such that 
\begin{itemize}
    \item[(i)] the left part of every trek starting in $A$ only takes edges in $G_1$, and
    \item[(ii)] the right part of every trek ending in $D$ only takes edges in $G_2$, 
\end{itemize}
then $\det(M) \neq 0$.
\end{theorem}
\begin{proof}
Let $\Pi$ be  a system of treks with no sided intersection from $A \cup B$ to $C \cup D$ such that Conditions (i) and (ii) are satisfied. Since the system has no sided intersection, the left system $\bfP_{A \cup B}$ from $\text{top}(\Pi)$ to $A \cup B$ and the right system $\bfP_{C \cup D}$ from $\text{top}(\Pi)$ to $C \cup D$ both have no intersection. Moreover, every path in the system $\bfP_{A \cup B}$ ending in $A$ only takes edges in $G_1$, and every path in the system $\bfP_{C \cup D}$ ending in $D$ only takes edges in $G_2$. We obtain by Lemma~\ref{lem:det-path-nonzero} that both $\det(L_{\text{top}(\Pi), A \cup B})$ and $\det(R_{\text{top}(\Pi), C \cup D})$ are nonzero. Thus, we conclude by Lemma~\ref{lem:cauchy-binet} that $\det(M)$ must also be nonzero.
\end{proof}

Theorem~\ref{thm:generalized-trek-separation} is the main technical tool for proving Theorem~\ref{thm:LSC}. However, it is only a sufficient condition for invertibility of matrices of the form~\eqref{eq:M-matrix}. The reverse direction of Lemma~\ref{lem:det-path-nonzero} and hence also the reverse direction of Theorem~\ref{thm:generalized-trek-separation} does not hold, as we we show in the next example.

\begin{example} \label{ex:counterexample-trek-separation}
Take $G$ to be the path $1 \rightarrow 2 \rightarrow 3 \rightarrow 4$ and consider the subgraph $G_1$ such that the only edge is given by $2\rightarrow 3$. Let $A=\{3\}, B=\{4\}$ and $ S=\{1,2\}$. Then,
\[
    L_{S,A \cup B} = 
    \begin{blockarray}{ccc}
    & 3 & 4 \\
    \begin{block}{c(cc)}
       1 & 0 & \lambda_{12} \lambda_{23} \lambda_{34} \\
       2 & \lambda_{23} & \lambda_{23}\lambda_{34} \\
    \end{block}
    \end{blockarray}.
\]
Clearly, $\det(L_{S,A \cup B}) \neq 0$ even though there is no system of of directed paths with no intersection from $S$ to $A \cup B$ such that every path ending in $A$ only has edges in $G_1$.
\end{example}

\section{Discussion}
In this work, we proposed a graphical criterion that provides a sufficient condition for rational identifiability of semi-direct effects. The criterion operates on models with arbitrary latent structure. Crucially, identifiability of semi-direct effects allows to transform the model into a simpler measurement model. Then, it might become possible to apply existing identification rules to identify causal effects even between latent variables; recall Example~\ref{ex:identifiying-latent-effects}.

Our work opens up some natural questions for further studies. Since the latent-subgraph criterion only provides a sufficient condition for rational identifiability, it is desirable to also obtain a powerful necessary condition in form of a graphical criterion. One approach is to compare the dimension of the image of the parametrization $\tau_G$ defined in~\eqref{eq:parametrization}, with the image of the map $\phi_G$ that maps the parameters $(\Lambda, \Phi)$ to the pair of matrices $(\oLambda, \Omega)$. If $\im(\tau_G) < \im(\phi_G)$, rational identifiability can not hold. Studying the images of the maps $\tau_G$ and $\phi_G$ amounts to studying the maximal rank of the corresponding Jacobian matrices.

Our main technical tool is an extension of the concept of trek separation~\citep{sullivant2010trek}. It can be used to certify that determinants of matrices with entries from different blocks, corresponding to treks in separate subgraphs, are generically nonzero. However, as we have seen in Example~\ref{ex:counterexample-trek-separation}, the condition in Theorem~\ref{thm:generalized-trek-separation} is not an ``if and only if'' condition. Finding such a condition is of potential interest in applications beyond determining rational identifiability. As shown by \citet{drton2020nested}, a characterization of the vanishing of the determinants of trek matrices involving subgraphs may lead to the discovery of novel, non-determinantal constraints that hold on every covariance matrix in the model. Such constraints are, for example, interesting for model equivalence and constraint-based statistical inference.

Finally, the integer linear program we set up in Section~\ref{sec:integer-linear-program} is a generalization of the usual maximum flow problem: If a flow ends at a certain node, then the flow is allowed to only take edges in a subgraph. We believe that solving Conjecture~\ref{conj:integer-solution}, whether the program can be solved in polynomial time, is of independent interest.

\section*{Acknowledgments}
This project has received funding from the European Research Council (ERC) under the European Union’s Horizon 2020 research and innovation programme (grant agreement No 883818) and from the German Federal Ministry of Education and Research and the Bavarian State Ministry for Science and the Arts.

\newpage
\begin{appendix}
\begin{center}
{ \LARGE\bf \noindent Appendix} 
\end{center}

\section{Proof of Main Result} \label{sec:proof-main-result}
\begin{proof}[Proof of Theorem~\ref{thm:LSC}]
Let $Y = \{y_1,\dots,y_{n+r}\}$ and $Z = \{z_1,\dots,z_r\}$ be as in the statement of the theorem and denote $\opa(v) = \{p_1,\dots,p_n\}$. We define two matrices $A\in\bR^{(n+r)\times n},B\in\bR^{(n+r)\times r}$ and a vector $c\in\bR^{n+r}$ as follows:
\[A_{ij} = \begin{cases}
\left[(I - \oLambda)^\top\Sigma\right]_{y_ip_j},&\text{ if }y_i\in\elr_{H_2,H_1}(Z\cup \{v\}),\\
\Sigma_{y_ip_j},&\text{ if }y_i\not\in\elr_{H_2,H_1}(Z\cup\{v\}),
\end{cases}\]
and
\[B_{ij} = \begin{cases}
\left[(I - \oLambda)^\top\Sigma(I - \oLambda)\right]_{y_iz_j},&\text{ if }y_i\in\elr_{H_2,H_1}(Z\cup \{v\}),\\
\left[\Sigma(I - \oLambda)\right]_{y_iz_j},&\text{ if }y_i\not\in\elr_{H_2,H_1}(Z\cup\{v\}),
\end{cases}\]
and
\[c_i = \begin{cases}
\left[(I - \oLambda)^\top\Sigma\right]_{y_iv},&\text{ if }y_i\in\elr_{H_2,H_1}(Z\cup\{v\}),\\
\Sigma_{y_iv},&\text{ if }y_i\not\in\elr_{H_2,H_1}(Z\cup\{v\}).
\end{cases}\]

We divide the proof of the theorem into $5$ separate steps.
\medskip

\textbf{Claim $1$.} The matrices $A$ and $B$ and the vector $c$ are rationally identifiable.  \\[0.1cm]
By assumption, all semi-direct effects  $u \rightsquigarrow w$ into a node $w \in Z\cup (Y\cap \elr_{H_2,H_1}(Z\cup\{v\}))$ are rationally identifiable.  Since only entries $\olambda_{uw}$ with $w \in Z\cup (Y\cap \elr_{H_2,H_1}(Z\cup\{v\}))$ appear in the definition of $A$, $B$ and $c$ we conclude that $A$, $B$ and $c$ are rationally identifiable (i.e.~rational functions of $\Sigma$).
\medskip

\textbf{Claim $2$.} There exists a subset $Y_Z \subseteq Y$ with $|Y_Z|=|Z|$ such that  $\det(\Omega_{Y_Z, Z}) \neq 0$ generically.  
\medskip

By assumption, there is a subset $Y_Z \subseteq Y$ such that there is a system of treks with no sided intersection from $Y_Z$ to $Z$ in the subgraph $G_{\lat}$. Since the system has no sided intersection, it follows from Proposition 3.4 in \citet{sullivant2010trek} that, generically, $\det(\Omega_{Y_Z, Z}) \neq 0$. 
\medskip

\textbf{Claim $3$.} Let $X \subseteq V\setminus (Z \cup \{v\})$ be the largest possible set such that $(H_1,H_2)$ trek separates $X$ and $Z \cup \{v\}$ in the subgraph~$G_{\lat}$. Then, generically, there exists a vector $\psi \in \mathbb{R}^r$ such that $\Omega_{X,Z} \cdot  \psi = \Omega_{X,v}$. 
\medskip

Recall that 
\begin{align*}
\Omega_{X,Z \cup \{v\}} &= \Phi_{X,Z \cup \{v\}} + \Lambda_{\cL,X}^{\top} (I -\Lambda_{\cL,\cL})^{-\top} \Phi_{\cL,\cL} (I -\Lambda_{\cL,\cL})^{-1} \Lambda_{\cL,Z \cup \{v\}} 
\end{align*}
Since $\Phi$ is diagonal and $X \cap (Z \cup \{v\}) = \emptyset$, we have that $\Phi_{X,Z \cup \{v\}} = 0$. Because we assumed that $(H_1,H_2)$ trek separates $X$ and $Z \cup \{v\}$ in the subgraph~$G_{\lat}$, trek separation \citep[Theorem 2.8]{sullivant2010trek} yields 
\[
\rk(\Omega_{X,Z \cup \{v\}}) \leq |H_1|+|H_2|=|Z|.
\]
On the other hand, $\Omega_{X,Z}$ generically has full column rank $r$ by Claim $2$ because $Y_Z \subseteq X$ and thus $\Omega_{Y_Z, Z}$ is a submatrix with full rank. This proves that, generically, there exists $\psi \in \mathbb{R}^r$ such that $\Omega_{X,Z}  \cdot \psi = \Omega_{X,v}$.
\medskip

\textbf{Claim $4$.} With the same $\psi\in\bR^r$ as in Claim $3$ we have that
\begin{align*}
  \begin{pmatrix} A & B \end{pmatrix}\cdot \begin{pmatrix} \oLambda_{\opa(v),v} \\ \psi\end{pmatrix} = c\,. 
\end{align*}

\noindent
We prove Claim 4 by considering each row indexed by $i \in \{1, \ldots, n+r\}$ separately. In particular, we split the proof into two parts depending on whether $y_i \in \elr_{H_2,H_1}(Z\cup \{v\})$ or $y_i \not\in \elr_{H_2,H_1}(Z\cup \{v\})$. First, consider any index $i$ such that $y_i \in \elr_{H_2,H_1}(Z\cup \{v\})$.  Then
\begin{align}
\nonumber
&\left[\begin{pmatrix} A & B \end{pmatrix}\cdot \begin{pmatrix} \oLambda_{\opa(v),v} \\ \psi\end{pmatrix} \right]_i\\
\nonumber
&=\left[(I - \oLambda)^\top\Sigma\right]_{y_i,\opa(v)}\cdot \oLambda_{\opa(v),v} + \left[(I - \oLambda)^\top\Sigma(I - \oLambda)\right]_{y_i,Z}\cdot \psi\\
\label{eq:id-equation}
&=\left[(I - \oLambda)^\top\Sigma\cdot \oLambda\right]_{y_iv} + \left[\O_{Y,Z}\cdot \psi\right]_i
\end{align}
because $\oLambda_{wv}=0$ unless $w\in\opa(v)$ and $ (I - \oLambda)^\top\Sigma(I - \oLambda)=\O $. Note that the set $Y$ is a subset of $X$ defined in Claim $3$ and therefore, with the same $\psi \in \mathbb{R}^r$, we have that $\Omega_{Y,Z} \cdot \psi = \Omega_{Y,v}$. It follows that
\[
\left[\O_{Y,Z}\cdot \psi\right]_i = \left[\O_{Y,v}\right]_i = \O_{y_iv}.
\]
Thus, we can rewrite equation \eqref{eq:id-equation} as
\begin{align*}
\left[\begin{pmatrix} A & B \end{pmatrix}\cdot \begin{pmatrix} \oLambda_{\opa(v),v} \\ \psi\end{pmatrix} \right]_i
&=\left[(I - \oLambda)^\top\Sigma\right]_{y_iv} - \left[(I - \oLambda)^\top\Sigma (I - \oLambda)\right]_{y_iv} + \O_{y_iv}\\
&=\left[(I - \oLambda)^\top\Sigma\right]_{y_iv} - \O_{y_iv} + \O_{y_iv}\\
&=c_i, 
\end{align*}
by the definition of $c$. Next, consider any index $i$ such that $y_i\not\in\elr_{H_2,H_1}(Z\cup\{v\})$. Then, for all $w\in Z\cup\{v\}$,  any trek from a node $w$ to $y_i$  must be of the form
\begin{equation} \label{eq:intersecting-trek}
w \leftarrow h \leftarrow \cdots \rightarrow h' \rightarrow u \rightarrow x_1 \rightarrow \dots \rightarrow x_m \rightarrow y_i  
\end{equation}
such that $w \leftarrow h \leftarrow \cdots \rightarrow h' \rightarrow u$ is a latent trek where either the left part contains a node in $H_2$ or the right part contains a node in $H_1$. We observe that each $u$ appearing in a trek as in \eqref{eq:intersecting-trek} is an element of $X$ (where $X$ is defined in Claim $3$). This can be seen as follows.  We have $u \not\in Z \cup \{v\}$ since, otherwise, there would exist the trek $u \rightarrow x_1 \rightarrow \dots \rightarrow x_m \rightarrow y_i$ which contradicts $y_i\not\in\elr_{H_2,H_1}(Z\cup\{v\})$. Now, assume that there exists a latent trek $w \leftarrow h \leftarrow \cdots \rightarrow h' \rightarrow u $ such that the left part of the trek does not intersect with $H_2$ and the right part of the trek does not intersect with $H_1$. But then $y_i$ is extended latent reachable from $Z \cup \{v\}$ by avoiding $(H_2,H_1)$, which again contradicts $y_i\not\in\elr_{H_2,H_1}(Z\cup\{v\})$.  We conclude that the pair $(H_1,H_2)$ trek separates $u$ and $Z\cup\{v\}$, which is equivalent to saying that $u \in X$.

On the other hand, for observed nodes $v,w \in V$ we have
\[
    \left[\Omega (I-\oLambda)^{-1}\right]_{v w}=\sum_{P \in \text{LST}_G(v,w)}\prod_{x \rightarrow y \in P} \lambda_{xy},
\]
where $\text{LST}_G(v,w)$ is the set of ``left subgraph treks'' from $v$ to $w$ in $G$, i.e., the set of treks such that the left part is a trek in $G_{\lat}$. For all $w \in Z \cup \{v\}$, it follows that  
\begin{equation} \label{eq:not_htr}
\left[\Omega (I-\oLambda)^{-1}\right]_{w y_i} = \Omega_{w,X} \left[(I-\oLambda)^{-1}\right]_{X,y_i}.
\end{equation}
Now, 
\begin{align}
\nonumber
\left[\begin{pmatrix} A & B \end{pmatrix}\cdot \begin{pmatrix} \oLambda_{\opa(v),v} \\ \psi\end{pmatrix} \right]_i
&=\Sigma_{y_i,\opa(v)}\cdot \oLambda_{\opa(v),v} + \left[\Sigma(I - \oLambda)\right]_{y_i,Z}\cdot \psi\\
\nonumber
&=\left[\Sigma \oLambda\right]_{y_i v} + \left[\Sigma(I - \oLambda)\right]_{y_i,Z}\cdot \psi \\ \label{eq:claim2_3}
&=\Sigma_{y_i v} - \left[\Sigma (I - \oLambda)\right]_{y_i v} + \left[\Sigma(I - \oLambda)\right]_{y_i,Z}\cdot \psi.
\end{align}
Using $\O = (I - \oLambda)^\top\Sigma(I - \oLambda)$ and applying \eqref{eq:not_htr} and Claim $3$, we have
\begin{align*}
\left[\Sigma(I - \oLambda)\right]_{y_i,Z}\cdot \psi &= 
\left[(I - \oLambda)^{-\top}\O\right]_{y_i,Z}\cdot \psi = \left[(I - \oLambda)^{-\top}\right]_{y_i,X} \O_{X,Z} \cdot \psi\\
&= \left[(I - \oLambda)^{-\top}\right]_{y_i,X} \O_{X,v}  = \left[(I - \oLambda)^{-\top} \O\right]_{y_i v} \\
&= \left[\Sigma(I - \oLambda)\right]_{y_i v}.
\end{align*}
According to equation \eqref{eq:claim2_3} and recalling the definition of $c$, we conclude that
\begin{align*}
\left[\begin{pmatrix} A & B \end{pmatrix}\cdot \begin{pmatrix} \L_{\pa_V(v),v} \\ \psi\end{pmatrix} \right]_i
&=\Sigma_{y_i v} = c_i.
\end{align*}

The theorem is now proven if the determinant of the matrix $\begin{pmatrix} A & B \end{pmatrix}$ is non-zero for generic parameter choices. In this case there exists a rational inverse of $\begin{pmatrix} A & B \end{pmatrix}$ and, generically, the equation system exhibited in Claim $4$ has a unique solution. This is addressed by our last claim:
\medskip

\textbf{Claim $5$.} The determinant of the matrix
$\begin{pmatrix} A & B \end{pmatrix}$
is generically non-zero.
\medskip

\noindent 
Denote $Y_1 = Y \cap \elr_{H_2,H_1}(Z\cup \{v\})$ and $Y_2=Y \setminus Y_1$. Then
\begin{align*}
    \begin{pmatrix} A & B \end{pmatrix}
    &= 
    \begin{blockarray}{ccc}
    & \opa(v) & Z \\
    \begin{block}{c(cc)}
       Y_1 & (I-\oLambda)^{\top} \Sigma & (I-\oLambda)^{\top} \Sigma (I-\oLambda) \\
       Y_2 & \Sigma & \Sigma (I-\oLambda) \\
    \end{block}
    \end{blockarray}    
\end{align*}
Recall that $\Lambda$ is given by
\[
    \Lambda = \begin{pmatrix}
         \Lambda_{\cO,\cO} & \Lambda_{\cO,\cL} \\
        \Lambda_{\cL,\cO} & \Lambda_{\cL,\cL}
    \end{pmatrix},
\]
and define the matrix
\[
    \Lambda_{\lat} = \begin{pmatrix}
        0 & 0 \\
        \Lambda_{\cL,\cO} & \Lambda_{\cL,\cL}
    \end{pmatrix}.
\]
Since we assumed that both $\det(I-\oLambda)$ and $\det(I-\Lambda_{\cL,\cL})$ are nonzero, we have that $\det(I-\Lambda) = \det(I-\oLambda) \cdot \det(I-\Lambda_{\cL,\cL})$ and $\det(I-\Lambda_{\lat}) = \det(I - \Lambda_{\cL,\cL})$ are nonzero. We make the following observations.
\begin{itemize}
    \item[(i)] $[(I - \Lambda)^{-1}]_{\cO,\cO} = (I - \oLambda)^{-1}$ and $[(I - \Lambda_{\lat})^{-1}]_{\cO,\cO} = I$,
    \item[(ii)] $[(I - \Lambda)^{-1}]_{\cL,\cO} = (I-\Lambda_{\cL,\cL})^{-1} \Lambda_{\cL,\cO} (I - \oLambda)^{-1}$ and $[(I - \Lambda_{\lat})^{-1}]_{\cL,\cO} =  (I-\Lambda_{\cL,\cL})^{-1} \Lambda_{\cL,\cO}$.
\end{itemize}
Observations (i) and (ii) imply that
\begin{align*}
(I-\oLambda)^{\top} \Sigma 
&= \Omega (I-\oLambda)^{-1} \\
&= \Lambda_{\cL,\cO}^{\top}  (I -\Lambda_{\cL,\cL})^{-\top} \Phi_{\cL,\cL} (I -\Lambda_{\cL,\cL})^{-1} \Lambda_{\cL,\cO} (I-\oLambda)^{-1} + \Phi_{\cO,\cO} (I-\oLambda)^{-1} \\
&= [(I - \Lambda_{\lat})^{-\top}]_{\cO,\cL} \Phi_{\cL,\cL} [(I - \Lambda)^{-1}]_{\cL,\cO} + [(I - \Lambda_{\lat})^{-\top}]_{\cO,\cO}\Phi_{\cO,\cO} [(I - \Lambda)^{-1}]_{\cO,\cO} \\
&= [(I - \Lambda_{\lat})^{-\top} \Phi (I - \Lambda)^{-1}]_{\cO,\cO}
\end{align*}
With similar calculations, it is easy to see that
\[
    \begin{pmatrix} A & B \end{pmatrix}
    = 
    \begin{blockarray}{ccc}
    & \opa(v) & Z \\
    \begin{block}{c(cc)}
       Y_1 & (I - \Lambda_{\lat})^{-\top} \Phi (I - \Lambda)^{-1} & (I - \Lambda_{\lat})^{-\top} \Phi (I - \Lambda_{\lat})^{-1} \\
       Y_2 & (I - \Lambda)^{-\top} \Phi (I - \Lambda)^{-1} & (I - \Lambda)^{-\top} \Phi (I - \Lambda_{\lat})^{-1} \\
    \end{block}
    \end{blockarray}.
\]
Now, observe that the matrix $\Lambda_{\lat}$ has sparsity pattern according to the graph $G_{\lat}$. That is, the $vw$-th entry of $\Lambda_{\lat}$ is given by $\lambda_{vw}$ if $v \rightarrow w$ in $D_{\lat}$, and zero otherwise. Therefore, it follows from  Theorem~\ref{thm:generalized-trek-separation} that the determinant of $\begin{pmatrix} A & B \end{pmatrix}$ is generically nonzero by using property (iii) in the definition of the latent-subgraph criterion.
\end{proof}

\section{Other Proofs} \label{sec:other-proofs}

\begin{proof}[Proof of Theorem~\ref{thm:linear-program}]
Consider $Y \subseteq Y_a$ and suppose that there is a system of directed paths $\Pi=\{\pi_1, \ldots, \pi_n\}$ with no intersection from $Y$ to $Z \cup P$ such that every path ending in $Z$ only has edges in $G_1$. W.l.o.g.~we assume that $\Pi$ has no cycles since, after removing all cycles from $\Pi$, we still have a system of directed paths $\Pi=\{\pi_1, \ldots, \pi_n\}$ with no intersection from $Y$ to $Z \cup P$ such that every path ending in $Z$ only has edges in $G_1$.

Now, we construct integer-valued vectors $(f,f^1)$ that maximize  $\text{Lp}(G,G_1,Z,P,Y_a)$.  Initialize $f_{uv}=f^1_{uv}=0$ for all $u \to v \in D_f$. For each path $\pi_k \in \Pi$, we set the corresponding values of $f$ or $f^1$ of the edges along the path to $1$. More precisely, suppose that the path $\pi_k$  starts at a node $y_k \in Y$ and terminates at a node $z_k \in Z$. For each edge $u \to v$ on the path $\pi_k$, we set $f^1_{uv}=1$, and, moreover, we set $f^1_{s y_k}=f^1_{z_k t}=1$. On the other hand, when the path $\pi_k$  starts at $y_k \in Y$ and terminates at $p_k \in P$, then, for each edge $u \to v$ on $\pi_k$, we set $f_{uv}=1$, and we set $f_{s y_k}=f_{p_k t}=1$. Clearly, it holds that $\sum_{z \in Z} f^1_{zt} + \sum_{p \in P} f_{pt} = |Z|+|P|$, and it remains to check that all constraints are satisfied. Note that (i) and (iii) are trivially satisfied by construction. For the constraints (ii) and (iv), we recall that the system of directed paths $\Pi$ has no intersection. It follows that each node appears at most once in the system. Hence, the constraints are satisfied.

Now, suppose $f^1$ and $f$ are integer-valued, $\sum_{z \in Z} f^1_{zt} + \sum_{p \in P} f_{pt} = |Z|+|P|$, and that the constraints (i)-(iv) are satisfied. Observe that the values of $f^1$ and $f$ must be in $\{0,1\}$.  Moreover, by constraint (ii), we can decompose the set of edges that have $f^1_{uv}=1$ or $f_{uv}=1$  into $n:=|Z|+|P|$ directed paths $\widetilde{\Pi}=\{\widetilde{\pi}_1, \ldots, \widetilde{\pi}_n \}$ from $s$ to $t$ with either $f^1_{uv}=1$ or $f_{uv}=1$ for all edges $u \to v$ along each path. By construction of $G_{\flow}$, each path is of the form
\[
    \widetilde{\pi}_k: s \rightarrow y_k \rightarrow \cdots \rightarrow w_k \rightarrow t
\]
with $y_k \in S$ and $w_k \in Z \cup P$. This defines the path system $\Pi=\{\pi_1, \ldots, \pi_n\}$ from $Y$ to $Z \cup P$, where each path $\pi_k$ is obtained from $\widetilde{\pi}_k$ by removing the edge $s \rightarrow y_k$ and the edge $w_k \rightarrow t$, i.e., 
\[
    \pi_k:  y_k \rightarrow \cdots \rightarrow w_k.
\]
By constraint (iii), if $w_k \in Z$, then the path $\pi_k$ can only have edges in $G_1$. Moreover, by constraints (ii) and (iv), the paths can not intersect. 
\end{proof}

\begin{proof}[Proof of Theorem~\ref{thm:sound-and-complete}]
Suppose that $G$ is LSC-identifiable. Then by Theorem~\ref{thm:LSC} 
there is a total ordering $\prec$ on $\cO$ such that $w \prec v$ whenever $w \in Z_v \cup (Y_v \cap \elr_{H_2,H_1}(Z_v \cup \{v\}))$ where $(Y^v, Z^v, H^v_1, H^v_2) \in 2^{\cO\setminus\{v\}}\times 2^{\cO\setminus \{v\}}\times 2^\cL\times 2^\cL$ is a tuple satisfying the LSC with respect to $v$. Hence, if $G$ is LSC-identifiable, we might label the elements $\{v_1, \ldots, v_d\}=\cO$ such that $v_1 \prec v_2 \prec \cdots \prec v_d$. 

Now we claim that after at most $k$ passes through the for loop in line $2$, all nodes $v_i$, $i \preceq k$, have already been added to the solved nodes $S$. We prove this by induction. Suppose that all nodes $v_1, \ldots, v_{k-1} \in S$ and we are now testing the $k$-th node $v_k$. Let $(Y^{v_k}, Z^{v_k}, H^{v_k}_1, H^{v_k}_2)$ be the triple satisfying the LSC with respect to $v_k$. At one point, we will visit the correct sets $H^{v_k}_1, H^{v_k}_2 \subseteq \cL$ in line $3$. If  $z \in Z^{v_k}$, then  $z \in S$ already and therefore $z \prec v_k$. It holds that $z \not\in \{v_k\} \cup \opa(v_k)$ by definition of the LSC and, by Lemma~\ref{lem:allowed-for-Z}, there is a latent trek from a node in $H_1^{v_k}$ to $z$ or there is a directed path in $G_{\lat}$ from a node in $H_2^{v_k}$ to $z$.
Thus, we will visit the correct set $Z^{v_k} \subseteq Z_a$ in line $5$. Now, take any $y \in Y^{v_k}$. By Lemma~\ref{lem-allowed-for-Y}, we have that $y \not\in \text{lr}_{H_2^{v_k},H_1^{v_k}}(Z^{v_k} \cup \{v_k\})$. Moreover, if $y \in \elr_{H^{v_k}_2, H^{v_k}_1}(Z^{v_k} \cup \{v_k\})$, then $y \prec v_k$ and thus $y \in S$. We conclude that $Y^{v_k} \subseteq Y_a$, and by Corollary~\ref{cor:LSC-linear-program} we will add $v_k$ to $S$. By induction, we obtain that $S=V$ after at most $|V|$ repetitions of line $2$ to $13$.

Conversely, suppose the algorithm finds $S=V$, and fix a node $v \in V$. When $v$ was added to $S$, there must have been sets  $H^v_1,H^v_2  \subseteq \cL$ and $Z^v \subseteq Z_a$ with $|Z^v|=|H_1^v|+|H_2^v|$ such that there is an integer solution of  $\text{Lp}(G^{\lp},G^{\lp}_{\lat},Z^v,\opa(v),Y_a)$  with optimal value $|Z^v| + |\opa(v)|$.  By Corollary~\ref{cor:LSC-linear-program}, we obtain that there is a set $Y^v \subseteq Y_a$ such that the tuple $(Y^v, Z^v, H^v_1, H^v_2)$ satisfies Condition~(iii) of the latent-subgraph criterion with respect to $v$. Moreover, the definitions of $Z_a$ and $Y_a$ imply that Conditions~(i) and (ii) of the latent-subgraph criterion are also satisfied. It remains to verify that all nodes $w \in Z^v \cup (Y^v \cap \elr_{H^v_2, H^v_1}(Z^v \cup \{v\}))$ were added to $S$ in the steps before. By construction, $Z_v \subseteq S$ at this stage of the algorithm. For all $w \in Y_a$ it holds that either $w \in S$ already or $w \not\in \elr_{H^v_2, H^v_1}(Z^v \cup \{v\})$. Thus, we have as well that $Y^v \cap\elr_{H^v_2, H^v_1}(Z^v \cup \{v\}) \subseteq S$ at this stage of the algorithm. Applying this reasoning to all $v \in V$, we see that $G$ is LSC-identifiable.
\end{proof}

\begin{proof}[Proof of Lemma~\ref{lem:equal-path-systems}]
Let $S=\{s_1, \ldots, s_n\}$, $A=\{a_1, \ldots, a_{\ell}\}$ and $B=\{a_{\ell+1},\ldots, a_n\}$ for $\ell \leq n$. Let $\bfP=\{P_1, \ldots, P_n\}$ such that the directed path $P_i$ has source $s_i$ and sink $a_i$. We write $\Psi=\{\psi_1,\ldots, \psi_n\}$ for the second system of directed paths, where each $\psi_i$ has source $s_i$ and sink $a_{\sigma(i)}$. Here, $\sigma$ is a permutation of the indices in $[n]$. For each $i \in [n]$, we have to show that $\psi_i=P_i$, which in particular implies $\sigma(i)=i$. 

First, consider the case where $P_i$ is the trivial path, that is, $s_i=a_i$. Since $s_i$ appears only once in the system $\bfP$, there is no edge in $\Pi$ that contains the node $s_i$. Moreover, since $\bfP(\lambda)=\Psi(\lambda)$, there is also no edge in $\Psi$ that contains the node $s_i$. Hence $\psi_i$ is also the trivial path, which implies that $a_i=s_i=a_{\sigma(i)}$. However, it must be that $\sigma(i)=i$ since $\bfP$ would have an intersection otherwise. Hence, we have that $\psi_i=P_i$.

Now, consider the case where $P_i$ is not the trivial path, that is, $P_i$ is of the form
\begin{equation} \label{eq:path-P_i}
    P_i: s_i \rightarrow z_1^i \rightarrow  z_2^i \rightarrow \cdots \rightarrow z_{k}^i = a_i.
\end{equation} 
Since $\bfP$ has no intersection and no path contains a cycle, the only edge in $\bfP$ containing $s_i$ is the edge $s_i \rightarrow z_1^i$. Since $\bfP(\lambda)=\Psi(\lambda)$, the only edge in $\Psi$ that contains the node $s_i$ is also the edge $s_i \rightarrow z_1^i$. Hence, the path $\psi_i$ must be of the form
\[
    \psi_i: s_i \rightarrow z_1^i \cdots .
\]
First, consider the case where the path $\psi_i$ consists only of the edge $s_i \rightarrow z_1^i$. Then $z_1^i \in A \cup B$, that is, $z_1^i=a_{\sigma(i)}$. But the system $\bfP$ has no intersection, which implies that we must have $z_1^i=a_i$. Since no path in $\bfP$ contains a cycle, we conclude that the path $P_i$ must also consist only of the edge $s_i \rightarrow z_1^i=a_i$. Hence, it must be that $P_i=\psi_i$. 

It remains to consider the case where the path $\psi_i$ has more than one edge, that is, it is of the form 
\[
    \psi_i: s_i \rightarrow z_1^i \rightarrow v \cdots .
\]
Since $\bfP(\lambda)=\Psi(\lambda)$, the edge $z_1^i \rightarrow v$ also has to be in $\bfP$. Since $\bfP$ has no intersection, there is no edge $a_i \rightarrow v$ in $\bfP$, and we obtain that $z_1^i \neq a_i$ and thus $k \geq 2$ in~\eqref{eq:path-P_i}. Now, observe that the only edge that starts in $z_1^i$ is the edge $ z_1^i \rightarrow  z_2^i$, that is, $v=z_2^i$. Otherwise, the system $\bfP$ either has an intersection or the path $P_i$ contains a cycle. It follows that $\psi_i$ is of the form 
\[
    \psi_i: s_i \rightarrow z_1^i \rightarrow z_2^i \cdots .
\]
We now proceed to add one edge after the other to the path $\psi_i$, applying the reasoning just used to all edges except the last one of $\psi_i$. For the last edge of $\psi_i$, we argue as in the case where $\psi_i$ contained only a single edge. This shows that $\psi_i$ must be equal to $P_i$, as displayed in \eqref{eq:path-P_i}.
\end{proof}

\begin{proof}[Proof of Lemma~\ref{lem:det-path-nonzero}]
By the definition of the matrix $L_{S, A \cup B}$, we have that the entry indexed by a node $s\in S$ and a node $a \in A$ is equal to 
\[
    (I-\Lambda_1)_{s,a} = \sum_{P \in \mathcal{P}_{G_1}(s,a)} P(\lambda),
\]
where $\mathcal{P}_{G_1}(s,a)$ is the set of directed paths from $s$ to $a$ that only take edges in the subgraph $G_1$. On the other hand, entries of $L_{S, A \cup B}$ that are indexed by $s \in S$ and $b \in B$ are equal to 
\[
    (I-\Lambda)_{s,b} = \sum_{P \in \mathcal{P}_{G}(s,b)} P(\lambda),
\]
where $\mathcal{P}_{G}(s,b)$ is the set of directed paths from $s$ to $b$ that can take edges in the entire graph $G$. Now, define the collection
\begin{align*}
    \Pi = \{\bfP : \,\, &\bfP \text{ is a system of directed paths from } S \text{ to } A \cup B \\
    & \text{ such that every path ending in } A \text{ only has edges in } G_1 \}.
\end{align*}
By the definition of the determinant, we obtain that
\[
    \det(L_{S, A \cup B}) = \sum_{\mathbf{P} \in \Pi} \sign(\mathbf{P}) \, \mathbf{P}(\lambda),
\]
where $\sign(\mathbf{P})$ is the sign of the permutation that writes the elements of $A \cup B$ in the order of their appearance as sinks of the directed paths in $\Psi$. By assumption, there exists a system $\bfP$ of directed paths with no intersection from $S$ to $A \cup B$ such that every path ending in $A$ only has edges in $G_1$. Observe that, after removing all cycles from $\bfP$, we still have a system of directed paths with no intersection from $S$ to $A \cup B$ such that every path ending in $A$ only has edges in $G_1$. Hence, we can assume w.l.o.g.~that $\bfP$ has no cycles. 

Now, take \emph{any} other system of directed paths $\Psi \in \Pi$, potentially with intersection, such that $\bfP(\lambda)=\Psi(\lambda)$. It follows from Lemma~\ref{lem:equal-path-systems} that the path systems must also coincide, that is, $\bfP=\Psi$. Therefore, the coefficient of the monomial $\bfP(\lambda)$ in $\det(L_{S, A \cup B})$ is given by $\sign(\bfP)$. In particular, $\det(L_{S, A \cup B})$ is not the zero polynomial/ power series.
\end{proof}

\section{Additional Examples} \label{sec:examples}

\subsection{Example for Parameter Matrices} \label{sec:example-parameter}
Consider the graph in Figure~\ref{fig:running-example} (a). The parameter matrices $\Lambda$ and $\Phi$ are given as
\[
    \Lambda = \begin{pmatrix}
        0 & 0 & 0 & 0 & 0 & \lambda_{v_1h_1} & 0 \\
        0 & 0 & 0 & 0 & 0 & 0 & 0 \\
        0 & 0 & 0 & \lambda_{v_3v_4} & 0 & 0 & 0 \\
        0 & 0 & 0 & 0 & 0 & 0 & 0 \\
        0 & 0 & 0 & 0 & 0 & 0 & 0 \\
        0 & 0 & 0 & 0 & 0 & 0 & \lambda_{h_1h_2} \\
        0 & \lambda_{h_2v_2} & \lambda_{h_2v_3} & \lambda_{h_2v_4} & \lambda_{h_2v_5} & 0 & 0
    \end{pmatrix}
\]
and
\[
     \Phi = \begin{pmatrix}
        \phi_{v_1} & 0 & 0 & 0 & 0 & 0 & 0 \\
        0 & \phi_{v_2} & 0 & 0 & 0 & 0 & 0 \\
        0 & 0 & \phi_{v_3} & 0 & 0 & 0 & 0 \\
        0 & 0 & 0 & \phi_{v_4} & 0 & 0 & 0 \\
        0 & 0 & 0 & 0 & \phi_{v_5} & 0 & 0 \\
        0 & 0 & 0 & 0 & 0 & \phi_{h_1} & 0 \\
        0 & 0 & 0 & 0 & 0 & 0 & \phi_{h_2}
    \end{pmatrix}.
\]
The semi-direct effect matrix $\oLambda$ is given as
\[
    \oLambda = \begin{pmatrix}
        0 & \lambda_{v_1h_1}\lambda_{h_1h_2}\lambda_{h_2v_2} & \lambda_{v_1h_1}\lambda_{h_1h_2}\lambda_{h_2v_3} & \lambda_{v_1h_1}\lambda_{h_1h_2}\lambda_{h_2v_4} & \lambda_{v_1h_1}\lambda_{h_1h_2}\lambda_{h_2v_5} \\
        0 & 0 & 0 & 0 & 0 \\
        0 & 0 & 0 & \lambda_{v_3v_4} & 0\\
        0 & 0 & 0 & 0 & 0 \\
        0 & 0 & 0 & 0 & 0 
    \end{pmatrix}.
\]
Since the observed covariance matrix $\Omega$ for the model given by the latent subgraph $G_{\lat}$ already becomes complicated, we only display the submatrix with rows and columns indexed by the subset $W = \{v_1, v_3, v_4\} \subset \cO$. It  is given by
\[
    \Omega_{W,W} =
    \begin{pmatrix}
        \phi_1 & 0 & 0\\
        \ast & \phi_{h_1}\lambda_{h_1h_2}^2\lambda_{h_2v_3}^2 + \phi_{h_2}\lambda_{h_2v_3}^2 + \phi_{v_3} & \phi_{h_1}\lambda_{h_1h_2}^2\lambda_{h_2v_3}\lambda_{h_2v_4} +  \phi_{h_2} \lambda_{h_2v_3}\lambda_{h_2v_4}\\
        \ast & \ast & \phi_{h_1}\lambda_{h_1h_2}^2\lambda_{h_2v_4}^2 + \phi_{h_2}\lambda_{h_2v_4}^2 + \phi_{v_4}
    \end{pmatrix},
\]
where the entries $\ast$ are given by symmetry. Now, the observable covariance matrix $\Sigma=(\sigma_{vw})_{v,w \in \cO}$ given by the ``original'' graph $G$ in Figure~\ref{fig:running-example} (a) also becomes complicated. We only display the entry $\sigma_{v_3v_4}$, which is given by the polynomial
\begin{align*}
    \sigma_{v_3v_4} 
    &= \phi_{v_1}\lambda_{v_1h_1}^2\lambda_{h_1h_2}^2\lambda_{h_2v_3}^2\lambda_{v_3v_4} \\
    &+ \phi_{v_1}\lambda_{v_1h_1}^2\lambda_{h_1h_2}^2\lambda_{h_2v_3}\lambda_{h_2v_4} \\
    &+ \phi_{h_1}\lambda_{h_1h_2}^2\lambda_{h_2v_3}^2\lambda_{v_3v_4} \\
    &+ \phi_{h_1}\lambda_{h_1h_2}^2\lambda_{h_2v_3}\lambda_{h_2v_4} \\
    &+ \phi_{h_2}\lambda_{h_2v_3}^2\lambda_{v_3v_4} \\
    &+ \phi_{h_2}\lambda_{h_2v_3}\lambda_{h_2v_4}\\
    &+ \phi_{v_3}\lambda_{v_3v_4}.
\end{align*}

\subsection{Identifiability of Effects between Latent Variables} \label{sec:latent-effects}
Consider the graph in Figure~\ref{fig:air-pollution-relabeled}, which is a relabeled version of the graph in Figure~\ref{fig:air-pollution}. We verified in Example~\ref{ex:id-running-example-2} that this graph is LSC-identifiable. This implies that the matrix $\Omega=(I - \oLambda)^{\top} \Sigma (I - \oLambda)$ is also rationally identifiable. We fix the variances $\phi_{h_i}=1$ for the latent variables $h_1, h_2, h_3$. Since $\Omega \in \cM(G_{\lat})$ is the observed covariance matrix in the model given by the latent subgraph $G_{\lat}$, it follows from the trek rule in~\eqref{eq:trek-rule} that
\begin{align*}
    & \frac{\omega_{v_2 v_6}\omega_{v_3 v_5}}{\omega_{v_2 v_3}\omega_{v_5 v_6} - \omega_{v_2 v_6}\omega_{v_3 v_5}} \\
    = \,\,\, & \frac{\lambda_{h_1 v_2} \lambda_{h_1 h_3} \lambda_{h_3 v_6} \lambda_{h_1 v_3} \lambda_{h_1 h_3} \lambda_{h_3 v_5}}{\lambda_{h_1 v_2} \lambda_{h_1 v_3} (\lambda_{h_3 v_5} \lambda_{h_3 v_6} + \lambda_{h_1 h_3}^2\lambda_{h_3 v_5} \lambda_{h_3 v_6}) - \lambda_{h_1 v_2} \lambda_{h_1 h_3} \lambda_{h_3 v_6} \lambda_{h_1 v_3} \lambda_{h_1 h_3} \lambda_{h_3 v_5}} \\
    = \,\,\, & \frac{\lambda_{h_1 h_3}^2 \lambda_{h_1 v_2} \lambda_{h_1 v_3} \lambda_{h_3 v_5} \lambda_{h_3 v_6}}{\lambda_{h_1 v_2} \lambda_{h_1 v_3} \lambda_{h_3 v_5} \lambda_{h_3 v_6}} \\
    = \,\,\, &  \lambda_{h_1 h_3}^2.
\end{align*}
Up to relabeling and taking the root, this is the same formula as in Equation~\eqref{eq:formula-latent-effect}. Note that the identification of $\lambda_{h_1 h_3}$ in the measurement model given by $G_{\lat}$ is also given by the identification rules in \citet[Chapter 8]{bollen1989structural}.
\end{appendix}

\bibliographystyle{apalike}
\bibliography{literature}

\end{document}